\documentclass[11pt,oneside]{article}
\usepackage{tikz}
\usepackage{amssymb,amsmath,latexsym,amsfonts,amscd,amsthm,multirow,colortbl,caption,array}
\usepackage{ctable}
\usepackage{upgreek}
\usepackage{wasysym}
\usepackage{mathrsfs}
\usepackage{fancyhdr,tabularx,cite,mathrsfs,graphicx}
\usepackage{authblk}
\usepackage[headings]{fullpage}
\usepackage{stmaryrd}
\usepackage[all]{xy}
\usepackage{stmaryrd,rotating}
\usepackage{pdflscape,longtable}
\usepackage{accents}
\usepackage[OT2,T1]{fontenc}
\usepackage[font=small]{caption}

\usepackage{setspace}

\usepackage{hyperref}

\newcommand{\bea}{\begin{eqnarray}} 
\newcommand{\eea}{\end{eqnarray}} 
\newcommand{\bee}{\begin{eqnarray*}} 
\newcommand{\eee}{\end{eqnarray*}} 
\newcommand{\al}{\begin{align*}} 
\newcommand{\eal}{\end{align*}} 
\newcommand{\be}{\begin{equation}} 
\newcommand{\ee}{\end{equation}} 
 
\newcommand{\bem}{\begin{pmatrix}} 
\newcommand{\eem}{\end{pmatrix}}

\def\a{\alpha}

\def\th{\theta}

\newcolumntype{R}{ >{$}r <{$}}
\newcolumntype{C}{ >{$}c <{$}}
\newcolumntype{L}{ >{$}l <{$}}
\newcolumntype{F}{>{\centering\arraybackslash}m{1.5cm}}

\newcommand{\mc}[1]{\mathcal{#1}}
\newcommand{\ms}[1]{\mathscr{#1}}

\newcommand{\comment}[1]{}

\DeclareSymbolFont{cyrletters}{OT2}{wncyr}{m}{n}\DeclareMathSymbol{\Sha}{\mathalpha}{cyrletters}{"58}

\newcommand{\RR}{{\mathbb R}}
\newcommand{\CC}{{\mathbb C}}
\newcommand{\PP}{{\mathbb P}}
\newcommand{\ZZ}{{\mathbb Z}}
\newcommand{\QQ}{{\mathbb Q}}
\newcommand{\HH}{{\mathbb H}}
\newcommand{\JJ}{{\mathbb J}}

\newcommand{\tr}{\operatorname{{tr}}}

\newcommand{\ex}{\operatorname{e}} 

\newcommand{\num}{\operatorname{num}}

\newcommand{\indo}{\iota}
\newcommand{\wh}{{\rm wh}}

\newcommand{\opt}{{\rm opt}}

\newcommand{\xmod}{{\rm \;mod\;}}
\newcommand{\Av}{\operatorname{Av}}

\newcommand{\Th}{\Theta}

\newcommand{\opc}{{\rm c}}

\newcommand{\ua}{\upalpha}

\newcommand{\SL}{\operatorname{\textsl{SL}}}      
\newcommand{\mpt}{\widetilde{\SL}_2}      
\newcommand{\GL}{{\textsl{GL}}}      

\newcommand{\sM}{\mathsf{M}}
\newcommand{\se}{\mathsf{e}}	
\newcommand{\sg}{\mathsf{g}}	
\newcommand{\sMa}{{\mathsf{M}_{11}}} 
\newcommand{\sMb}{{\mathsf{M}_{12}}}

\newcommand{\sG}{\mathsf{G}}

\newcommand{\GJ}{\textsl{G}^{\rm J}}
\newcommand{\barGJ}{\overline{\textsl{G}}^{\rm J}}

\newcommand{\GammaOJ}{\Gamma_0^{\rm J}}
\newcommand{\GammaooJ}{\Gamma_{\infty}^{\rm J}}

\newcommand{\ON}{\textsf{ON}}	
\renewcommand{\Th}{\textsf{Th}}		

\newcommand{\calH}{\mathscr{H}}

\newcommand{\HGH}{{H}^{\rm Hur}}	
\newcommand{\HCE}{{H}^{\rm Coh}}	
\newcommand{\sHGH}{\ms{H}^{\rm Hur}}	
\newcommand{\sHCE}{\ms{H}^{\rm Coh}}	
\newcommand{\sHR}{\ms{H}^{\rm Rad}}	

\newtheorem{thm}{Theorem}[subsection]
\newtheorem*{thm*}{Theorem}

\newtheorem{pro}[thm]{Proposition}

\theoremstyle{definition}

\theoremstyle{remark}
\newtheorem{rmk}[thm]{Remark}

\numberwithin{equation}{subsection}

\begin{document}

\setstretch{1.26}

\title{
\vspace{-35pt}
\textsc{\huge{ 
{C}lass {N}umbers, {C}ongruent {N}umbers and {U}mbral {M}oonshine
}}
    }

\renewcommand{\thefootnote}{\fnsymbol{footnote}} 


\renewcommand{\thefootnote}{\arabic{footnote}} 

\author[1,2]{Miranda C.\ N.\ Cheng\thanks{mcheng@uva.nl}\thanks{mcheng@gate.sinica.edu.tw}}
\author[2]{John F.\ R.\ Duncan\thanks{jduncan@gate.sinica.edu.tw}}
\author[3]{Michael H.\ Mertens\thanks{mmertens@math.uni-koeln.de}}

\affil[1]{Institute of Physics and Korteweg-de Vries Institute for Mathematics, University of Amsterdam, Amsterdam, the Netherlands.}
\affil[2]{Institute of Mathematics, Academia Sinica, Taipei, Taiwan.}
\affil[3]{Department Mathematik/Informatik, Abteilung Mathematik, Universit\"at zu K\"oln, Weyertal 86--90, D-50931 K\"oln, Germany.} 

\date{} 

\maketitle

\abstract{
In earlier work we initiated a program to study relationships between finite groups and arithmetic geometric invariants of modular curves in a systematic way. 
In the present work we continue this program, with a focus on the two smallest sporadic simple Mathieu groups. 
To do this we first elucidate a connection between a special case of umbral moonshine and the imaginary quadratic class numbers.
Then we use this connection to classify a distinguished set of modules for the smallest sporadic Mathieu group. 
Finally we establish consequences of the classification for the congruent number problem of antiquity.
}

\clearpage

\tableofcontents

\section{Introduction}\label{sec:int}

In earlier work \cite{cncsga} we initiated a program to systematically study interrelationships between representations of finite groups and arithmetic-geometric invariants. We also took a first step in this program 
by 
revealing families of such 
relationships for the 
cyclic 
groups of prime order. 
Since the 
cyclic 
groups 
of prime order
are the simplest simple groups, 
it is natural to ask if the module structures of op.\ cit.\ extend to larger groups, and if so with what consequences for arithmetic. In this work we 
answer these questions in 
a special case.

\subsection{Motivation}\label{sec:int-mtv}

Relationships between a distinguished representation of the sporadic simple group $\ON$ of O'Nan and the arithmetic of certain elliptic curves were 
illuminated in \cite{MR4291251} (see also \cite{2017NatCo...8..670D,MR4139238}), and an analogous story for the sporadic simple group $\Th$ of Thompson was told in \cite{MR4230542}.
This motivated us to initiate a program to study such phenomena systematically, and  
our first contribution \cite{cncsga} to this program is an analysis of the case that we cast the cyclic simple groups 
in the roles played by 
$\ON$ and $\Th$
in \cite{MR4291251} and \cite{MR4230542}.
As a result we obtain 
arithmetic relationships between distinguished infinite-dimensional virtual graded modules for cyclic groups of prime order and imaginary quadratic twists of the modular Jacobians of prime conductor. 

To put this more concretely we note that Theorem 4.1.1 of \cite{cncsga} guarantees the existence of 
an infinite-dimensional virtual graded $\ZZ/N\ZZ$-module 
\begin{gather}\label{eqn:int-mtv:WZNZ}
W=\bigoplus_D W_D, 
\end{gather}
for each prime $N$, with the property that 
\begin{gather}\label{eqn:int-mtv:dimWD}
\dim W_D=12\num\left(\frac{N+1}{6}\right)h(D)
\end{gather} 
when $D<0$ is a fundamental discriminant (i.e.\ the discriminant of $\QQ(\sqrt{D})$), 
where $\num(\a)$ denotes the numerator of a rational number $\a$ and $h(D)$ denotes the class number 
of $\QQ(\sqrt{D})$. 
Theorem 4.2.1 of op.\ cit.\ translates 
this property (\ref{eqn:int-mtv:dimWD})
into a congruence condition on $\dim W_D$ that obstructs the existence of infinite-order rational points on the $D$-twist 
of the Jacobian $J_0(N)$ of the modular curve $X_0(N)$, for suitably chosen $D$.

For example, $J_0(19)$
is the elliptic curve defined by the equation
	$y^2=x^3-12096x-544752$
(see e.g.\ \cite{lmfdb}), 
and the $D$-twist of $J_0(19)$ is
the elliptic curve defined by
\begin{gather}\label{eqn:int-mtv:J0(19)D}
y^2 = x^3 - 12096D^2x -  544752D^3.
\end{gather}
Taking $N=19$ in Theorem 4.2.1 of \cite{cncsga} we obtain that 
if $D<0$ is a fundamental discriminant such that 
$19$ is inert in the ring of integers of $\QQ(\sqrt{D})$, 
and 
\begin{gather}\label{eqn:int-mtv:hDcong}
	h(D)\not\equiv 0\xmod 3,
\end{gather}
then the curve 
(\ref{eqn:int-mtv:J0(19)D})
has only finitely many rational points. 
(Cf.\ Corollary 4.2.2 of op.\ cit., which is a counterpart to this statement for $N=11$.)

The reader may be interested to know how this compares to the results of the motivating works \cite{MR4291251} and \cite{MR4230542}. 
With this in mind we point out the following 
consequence of Theorem 1.3 of \cite{MR4291251}:
If 
$D<0$ is a fundamental discriminant such that 
$19$ is inert in the ring of integers of $\QQ(\sqrt{D})$, 
and 
\begin{gather}\label{eqn:int-mtv:dimWONDcong}
\dim W_D^\ON \not\equiv -24h(D) \xmod 19,
\end{gather}
where $W^\ON=\bigoplus_D W_D^\ON$ is the 
$\ON$-module 
considered in op.\ cit.\
(with $W^\ON_D$ here denoted by $W_{-D}$ there), 
then the curve 
(\ref{eqn:int-mtv:J0(19)D})
again has only finitely many rational points. 

By a similar token Theorem 1.3 of \cite{MR4230542} tells us that if 
$D<0$ is a fundamental discriminant such that 
$19$ is inert in the ring of integers of $\QQ(\sqrt{D})$, 
and
\begin{gather}\label{eqn:int-mtv:dimWThDcong}
\dim W_D^\Th \not\equiv 0 \xmod 19,
\end{gather}
where $W^\Th=\bigoplus_D W_D^\Th$ is 
any of the $\Th$-modules considered in op.\ cit.\ 
(with $W^\Th_D$ here denoted by $W_{-D}$ there), 
then the curve 
(\ref{eqn:int-mtv:J0(19)D})
yet again has only finitely many rational points. 

Note that the class number $h(D)$ does not appear in (\ref{eqn:int-mtv:dimWThDcong}). 
Tracing through the arguments of \cite{MR4230542} we see that this is because the graded dimension 
\begin{gather}\label{eqn:int-mtv:sumDdimWThDqD}
\sum_{D\leq 5} \dim W^\Th_Dq^{-D}
=6q^{-5}+O(q^3)
\end{gather} 
has vanishing constant term. 
In (\ref{eqn:int-mtv:hDcong}) we have the other extreme: We can omit $\dim W_D$ from (\ref{eqn:int-mtv:hDcong}) because it is a fixed multiple (\ref{eqn:int-mtv:dimWD}) of $h(D)$ by construction.

Theorem 4.2.1 of \cite{cncsga} furnishes obstruction statements similar 
to 
those of 
(\ref{eqn:int-mtv:hDcong}-\ref{eqn:int-mtv:dimWThDcong})
for each prime $N$. 
This demonstrates
that the results of \cite{MR4230542,MR4291251}, involving the sporadic simple groups of O'Nan and Thompson, are part of a more 
expansive picture.

It is natural to ask now how expansive this picture is.
To explore this general question we formulate the following more specific ones:
\begin{enumerate}
\item
Are the cyclic group module structures on the $W$ of (\ref{eqn:int-mtv:WZNZ}-\ref{eqn:int-mtv:dimWD}) restrictions of module structures for more sophisticated groups?
\item
If so, do these more sophisticated group-module structures 
on the $W$ of (\ref{eqn:int-mtv:WZNZ}-\ref{eqn:int-mtv:dimWD}) entail 
more sophisticated relationships to arithmetic geometry, than those formulated in (\ref{eqn:int-mtv:hDcong}-\ref{eqn:int-mtv:dimWThDcong}) above?
\end{enumerate}

In this work we 
answer these questions positively in
the ``simplest'' non-trivial case, by which we mean the case that $N=11$, since this is the smallest prime $N$ for which 
 $J_0(N)$ is not just a point.

\subsection{Methods}\label{sec:int-mth}

As in \cite{cncsga} 
our focus is on
virtual graded $\sG$-modules $W=\bigoplus_D W_D$, for $\sG$ a finite group, with the property that the {\em McKay--Thompson series}
\begin{gather}\label{eqn:int-mth:phiWsg}
	\phi^W_\sg(\tau,z):=\sum_{n,s\in\ZZ}\tr(\sg|W_{s^2-4n})q^ny^s
\end{gather}
defines an optimal holomorphic mock Jacobi form of weight $2$ and index $1$ 
and level $o(\sg)$
for each $\sg\in \sG$, 
once we set $q=e^{2\pi i \tau}$ and $y=e^{2\pi i z}$. We recall holomorphic mock Jacobi forms, and our notion of level, in \S~\ref{sec:prp-jac}.
The most important requirement is optimality, which we explain in detail in \S~\ref{sec:prp-opt}. 
In short it is the condition that the associated McKay--Thompson series $\phi^W_\sg$ (\ref{eqn:int-mth:phiWsg}) have a common fixed constant term, and have theta-coefficients $h^W_{\sg,r}$ (see (\ref{eqn:prp-jac:thetadecomp}) and (\ref{eqn:prp-opt:W}-\ref{eqn:prp-opt:hWsgr})) that vanish away from the infinite cusps of their invariance groups (cf.\ (\ref{eqn:prp-opt:phi_optimal})). In 
\cite{cncsga} the invariance group of $\phi^W_\sg$ is always assumed to be $\GammaOJ(o(\sg))$ (cf.\ (\ref{eqn:prp-jac:GammaooJ})). Here we allow for more general choices, and also allow characters.

A key tool in this work with no counterpart in \cite{cncsga} is a connection between 
meromorphic Jacobi forms in weights $1$ and $2$, which, via the methods of \cite{Dabholkar:2012nd}, may be used to connect suitable mock Jacobi forms in weights $1$ and $2$.
The McKay--Thompson series of umbral moonshine (see \cite{UM,MUM,MR3766220}) are naturally packaged as mock Jacobi forms of weight $1$, so there is a chance that we can use this connection to translate umbral moonshine into the setting of this work. 
As we demonstrate in \S~\ref{sec:mat-um3}, this works out especially well in  
the $\ell=3$ case of umbral moonshine, in that it yields for us a virtual graded module 
for the unique non-trivial double cover $2.\sMb$ of the sporadic simple Mathieu group $\sMb$, with McKay--Thompson series that have weight $2$ and index $1$, and that are optimal in the sense just sketched. Moreover, by restriction we recover from this optimal $2.\sMb$-module an optimal module for $\ZZ/11\ZZ$, such as 
in (\ref{eqn:int-mtv:WZNZ}-\ref{eqn:int-mtv:dimWD}).

As a result, through consideration of the specific case that $N=11$, we obtain a positive answer to the first question we formulated at the end of \S~\ref{sec:int-mtv}:
The cyclic group module structure on the $W$ of (\ref{eqn:int-mtv:WZNZ}-\ref{eqn:int-mtv:dimWD}) can arise as the restriction of something more intricate. 	

It is a curious fact that $\sMb$ contains two conjugacy classes of subgroups isomorphic to the smallest sporadic simple group, $\sMa$. 
To distinguish them, fix a non-trivial permutation representation $\sMb\to {\sf S}_{12}$ of $\sMb$ on $12$ points.
Then one of the classes 
is composed of the $12$ subgroups of $\sMb$ that fix one of the $12$ points.
The $12$ copies of $\sMa$ in the other class each act transitively on these $12$ points.
The group $\sMa$ has trivial Schur multiplier (cf.\ e.g.\ \cite{atlas}), so each of the $24$ copies of $\sMa$ in $\sMb$ pulls back to a direct product $2\times \sMa$ in $2.\sMb$.

Thus there are essentially two ways to restrict the optimal module for $2.\sMb$ in weight $2$ and index $1$ that comes to us from umbral moonshine to an optimal module of weight $2$ and index $1$ for $\sMa$. Since umbral moonshine naturally equips the umbral group $2.\sMb$ at $\ell=3$ with a permutation representation $2.\sMb\to {\sf S}_{12}$, it is natural to refer to these two optimal modules for $\sMa$ as {\em transitive} and {\em intransitive}, according as the corresponding composition
\begin{gather}\label{eqn:int-mth:M11toS12}
	\sMa\to 2.\sMb\to {\sf S}_{12}
\end{gather}
is transitive or not. From this perspective the transitive optimal module is distinguished by associating McKay--Thompson series with non-trivial characters to elements of order $4$ and $8$ in $\sMa$, while $\phi^W_\sg$ is a mock Jacobi form for $\GammaOJ(o(\sg))$ for all the $\sg$ arising from the intransitive optimal module. 
Thus the smallest sporadic simple group fits naturally into our framework in two different ways. We elucidate this 
further by making it our focus in \S\S~\ref{sec:mat-opt}-\ref{sec:mat-art}.

\subsection{Results}\label{sec:int-res}

In \S~4.1 of \cite{cncsga} we classified optimal virtual graded $\sG$-modules of weight $2$ and index $1$, for $\sG$ a cyclic group of prime order. In \S~4.2 of this work we carry out a directly similar classification but for $\sG$ the sporadic simple Mathieu group $\sMa$. In fact we establish two such classifications, corresponding to the two optimal modules for $\sMa$ that come to us from umbral moonshine, as we have just discussed. 

More specifically, we first determine the constant 
$\opc^\opt_{2,1}(\sG)$ 
and the lattice 
$\mc{L}^\opt_{2,1}(\sG)$ for $\sG=\sMa$, 
where $\opc^\opt_{2,1}(\sG)$ and $\mc{L}^\opt_{2,1}(\sG)$ are as defined in \S~\ref{sec:prp-opt} (this being the same as the definition in \cite{cncsga}). 
This classification includes the intransitive optimal module for $\sMa$, and is the content of Theorem \ref{thm:mat-opt:tor}. 

We then determine $\opc^\opt_{2,1}(\sG,\ua)$ and $\mc{L}^\opt_{2,1}(\sG,\ua)$, again for $\sG=\sMa$, but with $\ua$ encoding the multiplier systems that arise from the transitive optimal module for $\sMa$.
As we explain in \S~\ref{sec:prp-opt}, we define the constant $\opc^\opt_{2,1}(\sG,\ua)$ and the lattice $\mc{L}^\opt_{2,1}(\sG,\ua)$ analogously to $\opc^\opt_{2,1}(\sG)$ and $\mc{L}^\opt_{2,1}(\sG)$, respectively, but take into account the multiplier system data of $\ua$. The determination of $\opc^\opt_{2,1}(\sG,\ua)$ and $\mc{L}^\opt_{2,1}(\sG,\ua)$ here is the content of Theorem \ref{thm:mat-opt:tra}, and the classification this implies includes the transitive optimal module for $\sMa$.

It develops that 
$\opc^\opt_{2,1}(\sG)$ and $\opc^\opt_{2,1}(\sG,\ua)$ coincide for $\sG=\sMa$ and $\ua$ as arising from the transitive module. 
It is significant that this common value also coincides with $\opc^\opt_{2,1}(\sG)$ for $\sG=\ZZ/11\ZZ$, since this means that every optimal module 
for $\sG=\ZZ/11\ZZ$ extends to an optimal module for $\sMa$, and extends in two essentially different ways, 
depending on the triviality or not of the multiplier systems.

It is at least as significant that $\mc{L}^\opt_{2,1}(\sG)$ is the same for $\sG=\sMa$ and $\sG=\ZZ/11\ZZ$, but differs from $\mc{L}^\opt_{2,1}(\sG,\ua)$ for $\sG=\sMa$ and $\ua$ as in the transitive optimal module.
Indeed, it turns out that the latter lattice properly contains the former, because there is a non-zero cuspidal Jacobi form of weight $2$ and index $1$ for $\GammaOJ(8)$ that 
is consistent with $\ua$ as in the transitive optimal module.
It is the involvement of this cusp form that allows us to draw out a new kind of connection between groups and geometry in this work.

To explain this recall that a positive integer $n$ is called {\em congruent} if it occurs as the area of a right triangle with rational side lengths. 
That is, $n$ is congruent if there exist $a,b,c\in \QQ$ such that $a^2+b^2=c^2$ and $n=\frac12ab$. 
The problem of determining which positive integers are congruent goes back to Diophantus, and remains unsolved to this day. We refer to \cite{MR3713731} for a recent review.
In \S~\ref{sec:mat-art} we show that the representation theory of the transitive optimal module for $\sMa$ obstructs the existence of congruent numbers. Specifically, it follows from Theorem \ref{thm:mat-art:cngnmb} that 
if $W=\bigoplus_D W_D$ is the transitive optimal module for $\sMa$, 
and 
if $n$ is square-free and congruent to $3$ modulo $24$, 
then $n$ is not a congruent number whenever $W_{-n}$ contains the unique $55$-dimensional irreducible $\sMa$-module (cf.\ Table \ref{tab:chars:m11}) with non-zero multiplicity.

Thus we also answer in the affirmative, at least for $N=11$, the second question that we formulated the end of \S~\ref{sec:int-mtv}: 
More intricate group-module structures 
on the $W$ of (\ref{eqn:int-mtv:WZNZ}-\ref{eqn:int-mtv:dimWD}) can entail 
more intricate relationships to arithmetic geometry, than those formulated in (\ref{eqn:int-mtv:hDcong}-\ref{eqn:int-mtv:dimWThDcong}).

To conclude we mention that 
we obtain further positive answers to the questions of \S~\ref{sec:int-mtv}
in the forthcoming work \cite{hurm23},
by showing that 1) the optimal $\ZZ/N\ZZ$-modules of (\ref{eqn:int-mtv:WZNZ}-\ref{eqn:int-mtv:dimWD}) extend to modules for the second largest sporadic simple Mathieu group $\sM_{23}$ in the case that $N=23$, and 2) these modules entail 
interdependencies between arithmetic-geometric invariants of elliptic curves with coprime conductors. 

\subsection{Overview}

The structure of this article is as follows. 
We present a guide to the specialized notation that we use in \S~\ref{sec:not}. Then in \S~\ref{sec:prp} we prepare for the statements and proofs of our main results. 
Specifically, we review generalized Hurwitz class numbers, and some related notions, in \S~\ref{sec:prp-cln}, and review mock and meromorphic Jacobi forms in \S~\ref{sec:prp-jac}. 
We discuss the notion of optimality we use, in general weight and index, in \S~\ref{sec:prp-opt}, and describe it in more detail in the special case of weight $2$ and index $1$ in \S~\ref{sec:prp-fcs}. 
Our main results appear in \S~\ref{sec:mat}. We first explain the connection between class numbers and umbral moonshine in \S~\ref{sec:mat-um3}. We then use this to classify optimal modules of weight $2$ and index $1$ for $\sMa$, in two different ways, in \S~\ref{sec:mat-opt}. Finally we use the optimal modules of \S~\ref{sec:mat-opt} to relate the representation theory of $\sMa$ to the congruent number problem in \S~\ref{sec:mat-art}.

\section*{Acknowledgements}

We are grateful to Lea Beneish, Mathew Emerton, Maryam Khaqan, Kimball Martin, K.~Ono, Preston Wake, Eric Zhu and David Zureick-Brown for helpful discussions on topics closely related to the content of this work.
The work of M.C.\ was supported by the National Science and Technology Council of Taiwan (110-2115-M-001-018-MY3), and by a Vidi grant (number 016.Vidi.189.182) from the Dutch Research Council (NWO).
The work of J.D.\ was supported in part by the U.S.\ National Science Foundation (DMS 1601306), the Simons Foundation (\#708354), and the National Science and Technology Council of Taiwan (111-2115-M-001-001-MY2).

\section{Notation}\label{sec:not}

\begin{footnotesize}

\begin{list}{}{
	\itemsep -1pt
	\labelwidth 23ex
	\leftmargin 13ex	
	}
	
\item
[$\Av^{(m)}$]
The averaging operator of index $m$. 
See (\ref{eqn:prp-jac:Avm}).

\item
[$\ua$]
An assignment of discrete groups $\Gamma_\sg<\GJ$ and characters $\rho_\sg:\Gamma_\sg\to \CC^*$ to the elements $\sg$ of a finite group $\sG$.
See (\ref{eqn:prp-opt:alpha}) and cf.\ (\ref{eqn:prp-opt:default_alpha}).

\item
[$\opc$]
The negative of the constant term of a holomorphic mock Jacobi form. See (\ref{eqn:prp-opt:phiequalsminusc}).

\item
[$\opc^\opt_{k,m}(\sG)$]
A shorthand for $\opc^\opt_{k,m}(\sG,\ua)$, for the default choice (\ref{eqn:prp-opt:default_alpha}) of $\ua$.

\item
[$\opc^\opt_{k,m}(\sG,\ua)$]
An invariant we attach to a finite group $\sG$ and choice of $\ua$ in \S~\ref{sec:prp-opt}.
See (\ref{eqn:prp-opt:opcopt}).

\item
[$C_N(D)$]
A shorthand for $C^{(\sMa,\ua)}_\sg(D)$, when $o(\sg)=N$, in the proof of Theorem \ref{thm:mat-art:cngnmb}.

\item
[$C_{\varphi_{8|4}}(D)$]
A coefficient 
in the Fourier expansion of $\varphi_{8|4}$. 
See the proof of Theorem \ref{thm:mat-art:cngnmb}.

\item
[$\delta_N$]
The class function on a group $\sG$ defined by setting $\delta_N(\sg):=1$ in case $o(\sg)=N$, and $\delta_N(\sg)=0$ otherwise. Cf.\ (\ref{eqn:mat-opt:11d}) and (\ref{eqn:mat-opt:8d}).

\item
[$\se$]
The identity element of a finite group $\sG$.
See \S~\ref{sec:prp-opt}.

\item
[$\ex(\cdot)$]
We set $\ex(x):=e^{2\pi i x}$.

\item
[$E$]
The elliptic curve over $\QQ$ defined by $y^2=x^3-x$. 
See the proof of Theorem \ref{thm:mat-art:cngnmb}. 

\item
[$E\otimes D$]
The $D$-twist of the elliptic curve $E$. 
See the proof of Theorem \ref{thm:mat-art:cngnmb}. 

\item
[$\phi^W_\sg$]
The McKay--Thompson series 
associated to
the action of 
$\sg$ on 
$W$. 
See (\ref{eqn:prp-opt:phiWsg}) and (\ref{eqn:prp-fcs:phiWsg}).

\item
[$\varphi_{11}$]
A certain cuspidal Jacobi form of weight $2$ and index $1$ for $\GammaOJ(11)$. 
See (\ref{eqn:mat-opt:varphi11}).

\item
[$\varphi_{8|4}$]
A certain cuspidal Jacobi form of weight $2$ and index $1$ for $\GammaOJ(32)$. 
See (\ref{eqn:mat-opt:varphi84}).

\item
[$\sg$]
An element in a finite group $\sG$.
See \S~\ref{sec:prp-opt}.

\item
[$\sG$]
A finite group.
See \S~\ref{sec:prp-opt}.

\item
[$\GJ$]
The Jacobi Group. Cf.\ (\ref{eqn:prp-jac:mltGJ}).

\item
[$\Gamma$]
A subgroup of $\GJ$. Cf.\ (\ref{eqn:prp-jac:GammaooJ}).

\item[$\GammaOJ(N)$]
A group of the form $\Gamma_0(N)\ltimes \ZZ^2$. Cf.\ (\ref{eqn:prp-jac:GammaooJ}).

\item
[$h^W_\sg$]
The vector-valued function that takes the $h^W_{\sg,r}$ as its components.
See (\ref{eqn:prp-opt:phiWsg}).

\item
[$h^W_{\sg,r}$]
The theta-coefficients of $\phi^W_\sg$.
See (\ref{eqn:prp-opt:hWsgr}).

\item
[$\HGH(D)$]
The Hurwitz class number of discriminant $D$. 
See (\ref{eqn:prp-cln:HHurD}).

\item
[$\HGH_N(D)$]
The generalized Hurwitz class number of level $N$ and discriminant $D$.
See (\ref{eqn:prp-cln:HHurND}-\ref{eqn:prp-cln:HN0}).

\item
[$\HCE_N(D)$]
A coefficient of 
the Cohen--Eisenstein series $\sHCE_N$. 
Cf.\ (\ref{eqn:prp-cln:HCohN}--\ref{eqn:prp-cln:msHCohN}).

\item
[$\sHGH$]
A shorthand for $\sHGH_1$.
Cf.\ (\ref{eqn:prp-cln:msHN}).

\item
[$\sHGH_N$]
A holomorphic mock Jacobi form defined by the $\HGH_N(D)$.
See (\ref{eqn:prp-cln:msHN}).

\item
[$\sHCE_N$]
A holomorphic Jacobi form 
defined by the $\HCE_N(D)$ for $N$ prime.
See (\ref{eqn:prp-cln:msHCohN}).

\item
[$\indo(N)$]
The index of $\Gamma_0(N)$ as a subgroup of $\Gamma_0(1)=\SL_2(\ZZ)$. 
Cf.\ (\ref{eqn:prp-cln:HN0}).

\item
[$J_{k,m}(\Gamma,\rho)$]
The holomorphic Jacobi forms of weight $k$ and index $m$ for $\Gamma$ with character $\rho$. 
See \S~\ref{sec:prp-jac}.

\item
[$\JJ_{k,m}(\Gamma,\rho)$]
The holomorphic mock Jacobi forms of weight $k$ and index $m$ for $\Gamma$ with character $\rho$. 
See \S~\ref{sec:prp-jac}.

\item
[$\JJ^\wh_{k,m}(\Gamma,\rho)$]
The weakly holomorphic mock Jacobi forms of weight $k$ and index $m$ for $\Gamma$ with character $\rho$. 
See \S~\ref{sec:prp-jac}.

\item
[$L_{k,m}(\sG)$]
A certain lattice. 
See (\ref{eqn:prp-opt:LkmsG}).

\item
[$\mc{L}^\opt_{k,m}(\sG)$]
A shorthand for $\mc{L}^\opt_{k,m}(\sG,\ua)$ for the default choice (\ref{eqn:prp-opt:default_alpha}) of $\ua$.

\item
[$\mc{L}^\opt_{k,m}(\sG,\ua)$]
The lattice of $0$-optimal virtual graded modules of weight $k$ and index $m$ for $(\sG,\ua)$. 
Cf.\ Proposition \ref{pro:prp-opt:mcWropmcLrop}.

\item
[$\psi^F$]
The finite part of a meromorphic Jacobi form $\psi$.
See (\ref{eqn:prp-jac:polarfinite}).

\item
[$\psi^P$]
The polar part of a meromorphic Jacobi form $\psi$.
See (\ref{eqn:prp-jac:polarfinite}).

\item
[$\psi^{(1)}$]
A certain meromorphic Jacobi form of weight $2$ and index $1$. 
See (\ref{eqn:prp-jac:psi1}) and (\ref{eqn:mat-um3:psi1}).

\item
[$\psi^{(1)}_\sg$]
A certain meromorphic Jacobi form of weight $2$ and index $1$, generally with level and character, for $\sg\in 2.\sMb$. 
See (\ref{eqn:mat-um3:psi1g}).

\item
[$\psi^{(3)}$]
A certain meromorphic Jacobi form of weight $1$ and index $3$. 
See (\ref{eqn:mat-um3:psi3}).

\item
[$\psi^{(3)}_\sg$]
The meromorphic Jacobi form of weight $1$ and index $3$ attached to $\sg\in 2.\sMb$ by umbral moonshine at $\ell=3$.
See (\ref{eqn:mat-um3:psi1g}).

\item
[$q$]
We set $q:=\ex(\tau)$ for $\tau\in \HH$.

\item[$\mc{Q}_N(D)$]
A set of binary quadratic forms with integer coefficients. 
See \S~\ref{sec:prp-cln}. 

\item
[$R(\sG)$]
The Grothendieck group of finitely generated $\CC\sG$-modules.
Cf.\ (\ref{eqn:prp-opt:LkmsG}).

\item
[$R(\sG)_{(N)}$]
The subgroup of $R(\sG)$ composed of virtual $\sG$-modules $V$ such that $\tr(\sg|V)=0$ unless $o(\sg)=N$. 
Cf.\ (\ref{eqn:mat-opt:ZZvarphi11otimesRsG11}) and (\ref{eqn:mat-opt:ZZvarphi84otimesRsG8plus}).

\item
[$\rho$]
A character of a subgroup $\Gamma$ of $\GJ$ with level. Cf.\ (\ref{eqn:prp-jac:kmrhoaction}).

\item
[$\varrho_m$]
A certain unitary representation of $\mpt(\ZZ)$. 
Cf.\ (\ref{eqn:prp-jac:barvarrhomtheta}).

\item
[$S_{k,m}(\Gamma,\rho)$]
The cuspidal Jacobi forms of weight $k$ and index $m$ for $\Gamma$ with character $\rho$. 
See \S~\ref{sec:prp-jac}.

\item
[$S_{k,m}(\Gamma,\rho)_\ZZ$]
The $\phi\in S_{k,m}(\Gamma,\rho)$ with rational integer Fourier coefficients.
See \S~\ref{sec:prp-jac}.

\item
[$\mpt(\ZZ)$]
The metaplectic double cover of $\SL_2(\ZZ)$. 
Cf.\ (\ref{eqn:prp-jac:barvarrhomtheta}).

\item
[$\theta_m$]
The vector-valued function that takes the $\theta_{m,r}$ 
as its components.
Cf.\ (\ref{eqn:prp-jac:barvarrhomtheta}).

\item
[$\theta_{m,r}$]
The theta series defined by the positive-definite even lattices of rank $1$.
See (\ref{eqn:prp-jac:thetamr}).

\item
[$\vartheta_j(\tau,z)$]
A Jacobi theta function. See (\ref{eqn:prp-jac:vartheta1}) and (\ref{eqn:mat-um3:theta}).

\item
[$\mc{W}^\opt_{k,m}{(\sG)}$]
A shorthand for $\mc{W}^\opt_{k,m}{(\sG,\ua)}$ for the default choice (\ref{eqn:prp-opt:alpha_condition}) of $\ua$.

\item
[$\mc{W}^\opt_{k,m}{(\sG)}_\opc$]
A shorthand for $\mc{W}^\opt_{k,m}{(\sG,\ua)}_\opc$ for the default choice (\ref{eqn:prp-opt:alpha_condition}) of $\ua$.

\item
[$\mc{W}^\opt_{k,m}{(\sG,\ua)}$]
The set of optimal virtual graded modules of weight $k$ and index $m$ for $(\sG,\ua)$. 
Cf.\ (\ref{eqn:prp-opt:mcWoptkmsGua}).

\item
[$\mc{W}^\opt_{k,m}{(\sG,\ua)}_\opc$]
The set of $\opc$-optimal virtual graded modules of weight $k$ and index $m$ for $(\sG,\ua)$. 
Cf.\ (\ref{eqn:prp-opt:mcWoptkmsGua}).

\item
[$\chi_j$]
An irreducible character of $\sMa$. See Table \ref{tab:chars:m11}. 

\item
[$y$]
We set $y:=\ex(z)$ for $z\in\CC$.

\end{list}

\end{footnotesize}

\section{Preparation}\label{sec:prp}

This section is a counterpart to \S~3 in \cite{cncsga}, and offers preparation for the arguments that will appear in \S~\ref{sec:mat}.
We briefly review generalized Hurwitz class numbers in \S~\ref{sec:prp-cln}.
The Jacobi forms we consider here are more general than those that appear in op.\ cit., as we discuss in \S~\ref{sec:prp-jac}. 
The notion of optimality is correspondingly more general, and we explain it in \S~\ref{sec:prp-opt}. 
We work with Jacobi forms of general weight and index in \S\S~\ref{sec:prp-jac}-\ref{sec:prp-opt}, for the sake of future applications. 
We specialize to the situation of weight $2$ and index $1$, being the setting of our main results, in \S~\ref{sec:prp-fcs}.

\subsection{Class Numbers}\label{sec:prp-cln}

As in \cite{cncsga} 
we let $\mc{Q}_N(D)$ denote the set of integer coefficient binary quadratic forms $Q(x,y)=Ax^2+Bxy+Cy^2$ of discriminant $D:=B^2-4AC$ with 
$A\equiv 0 \xmod N$.
We equip each $\mc{Q}_N(D)$ with a right action by the group 
\begin{gather}\label{eqn:prp-cln:Gamma0N}
\Gamma_0(N) :=
\left.\left\{ 
\left(\begin{smallmatrix} a&b\\c&d\end{smallmatrix}\right)\in\SL_2(\ZZ)\;\right|\; c\equiv0\xmod N
\right\}
\end{gather}
in the usual way, 
and we write $\Gamma_0(N)_Q$ for the stabilizer in $\Gamma_0(N)$ of a particular $Q\in \mc{Q}_N(D)$. 

Note that $\mc{Q}_N(D)$ is non-empty if and only if $D$ is a quadratic residue modulo $4N$. 
Say that $D\in\ZZ$ is a {\em discriminant} if $\mc{Q}_1(D)$ is not empty (i.e.\ $D$ is congruent to $0$ or $1$ modulo $4$), and say that $D$ is a {\em fundamental discriminant} if $D$ is the discriminant of the number field $\QQ(\sqrt{D})$. (This is the same as saying that either $D\equiv 1\xmod 4$ and $D$ is square-free, or $D=4d$ where $d$ is square-free and not a quadratic residue modulo $4$.) 

We restrict to negative discriminants in this work. 
For $D$ a negative fundamental discriminant write $h(D)$ for the class number of $\QQ(\sqrt{D})$. 
Then $h(D)$ is just the number of orbits of $\Gamma_0(1)=\SL_2(\ZZ)$ in its action 
on $\mc{Q}_1(D)$.
The {\em generalized Hurwitz class numbers}, denoted $\HGH_N(D)$, are defined for $D<0$ by 
setting
\begin{gather}\label{eqn:prp-cln:HHurND}
\HGH_N(D):=\sum_{Q\in\mc{Q}_N(D)/\Gamma_0(N)}\frac1{\#\Gamma_0(N)_Q}
\end{gather}
when $\mc{Q}_N(D)$ is not empty, and $\HGH_N(D):=0$ otherwise.
Therefore, the $\HGH_N(D)$ are closely related to the usual class numbers $h(D)$. Indeed, the {\em Hurwitz class number}
\begin{gather}\label{eqn:prp-cln:HHurD}
\HGH(D):=\HGH_1(D)
\end{gather} 
coincides with $h(D)$ if $D$ is fundamental and less than $-4$, while for the remaining negative fundamental discriminants we have $\HGH(-3)=\frac13$ and $\HGH(-4)=\frac12$.
For a more general statement we note from Lemma 3.1.2 of \cite{cncsga} that
\begin{gather}\label{eqn:prp-cln:HHurND-Nprime}
	\HGH_N(D)=\left(1+\left(\frac DN\right)\right)\HGH(D)
\end{gather}
for $D$ negative and fundamental when $N$ is prime, where $\left(\frac{\cdot}{\cdot}\right)$ denotes the Kronecker symbol (see e.g.\ p.\ 503 of \cite{MR909238} for the definition). 
In particular $\HGH_N(D) = 2h(D)$ when $D<-4$ and $D$ is a square modulo $4N$.

As in \cite{cncsga} our main motivation for the definition (\ref{eqn:prp-cln:HHurND}) is that it gives rise to holomorphic mock Jacobi forms (cf.\ Proposition \ref{pro:prp-fcs:msHHurNmockmsHCohNmod}). 
With this in mind we define 
\begin{gather}\label{eqn:prp-cln:HN0}
\HGH_N(0):=-\frac{1}{12}
\indo(N)
\end{gather}
where $\indo(N):=[\Gamma_0(1):\Gamma_0(N)]$, 
and set $\HGH_N(D):=0$ when $D>0$. 
We also define $q:=\ex(\tau)$ for $\tau$ in the upper half-plane, $\HH:=\{\tau\in \CC\mid \Im(\tau)>0\}$, and $y:=\ex(z)$ for $z\in \CC$, where $\ex(x):=e^{2\pi i x}$.
The generating functions
\begin{gather}\label{eqn:prp-cln:msHN}
\sHGH_N(\tau,z):=\sum_{n,s\in\ZZ}\HGH_N(s^2-4n)q^ny^s
\end{gather}
and $\sHGH:=\sHGH_1$ will play an important role in what follows.

In \S~\ref{sec:mat-art} we will also employ
Cohen--Eisenstein series, whose coefficients furnish a further variation on the class numbers $h(D)$ (cf.\ \cite{MR0382192}). 
To define these series recall that for $D$ an arbitrary discriminant we have $D=f^2D_0$, where $D_0$ is the discriminant of $\QQ(\sqrt{D})$ and $f$ is the {\em conductor} of the order 
\begin{gather}
\mathcal{O}_D:=\ZZ\left[\frac{D+\sqrt{D}}2\right].
\end{gather} 
Next choose a prime $N$, and for $D=f^2D_0$ as above
let $f'$ be the largest factor of $f$ that is coprime to $N$, and set $D':=(f')^2D_0$. 
Then we may follow \S~1 of \cite{MR894322} in defining $\HCE_N(D)$ for $N$ prime and $D$ a negative discriminant
by setting 
\begin{gather}
\HCE_N(D):=
\begin{cases}\label{eqn:prp-cln:HCohN}
	0&\text{ if $N$ splits in $\mathcal{O}_{D'}$,}\\
	\frac12\HGH(D')&\text{ if $N$ is ramified in $\mathcal{O}_{D'}$, }\\ 
	\HGH(D')&\text{ if $N$ is inert in $\mathcal{O}_{D'}$.}
\end{cases}
\end{gather}
For convenience we also set $\HCE_N(0):=\frac{N-1}{24}$, and set $\HCE_N(D):=0$ when $D$ is positive, or a negative integer that is not a discriminant. 
The {\em Cohen--Eisenstein series} of (prime) level $N$ are now defined by setting
\begin{gather}\label{eqn:prp-cln:msHCohN}
	\sHCE_N(\tau,z):=\sum_{n,s\in\ZZ} \HCE_N(s^2-4n)q^{n}y^s,
\end{gather}
where $q$ and $y$ are as in (\ref{eqn:prp-cln:msHN}).

For $N$ prime the Cohen--Eisenstein series $\sHCE_N$ is related to the generalized Hurwitz class number generating functions $\sHGH$ and $\sHGH_N$ by the formula $\sHGH=\sHCE_N+\frac12\sHGH_N$. This is the content of Lemma 3.1.1 of \cite{cncsga}. Lemma 3.1.2 of op.\ cit.\ also gives us a counterpart 
\begin{gather}\label{eqn:prp-cln:HCohND-Nprime}
	\HCE_N(D)=\frac12\left(1-\left(\frac DN\right)\right)\HGH(D),
\end{gather}
for $D$ negative and fundamental, to (\ref{eqn:prp-cln:HHurND-Nprime}).

\subsection{Jacobi Forms}\label{sec:prp-jac}

Here we explain our conventions on mock Jacobi forms. 
We assume some familiarity with the basic definitions, suggest 
\S~3.1 of \cite{MR4127159} and \cite{MR781735} for more on Jacobi forms, and suggest
\S~3.2 of \cite{MR4127159} and \S~7.2 of \cite{Dabholkar:2012nd} for more on mock Jacobi forms. 
We will also make use of meromorphic Jacobi forms. We refer to \S~2 of \cite{MR3995918} and \S~8 of \cite{Dabholkar:2012nd} for more background on these.

In \S~\ref{sec:mat-um3} we will require 
a 
notion of level for
discrete subgroups of the {Jacobi Group},  
\begin{gather}\label{eqn:prp-jac:GJ}
\GJ=\SL_2(\RR)\ltimes (\RR^2\cdot S^1).
\end{gather}
To explain 
what we mean 
by
this 
we first realize 
$\GJ$ 
as the set of 
triples 
$(\gamma,(\lambda,\mu),\varsigma)$, where $\gamma\in \SL_2(\RR)$, and $(\lambda,\mu)\in \RR^2$, and $\varsigma\in\CC$ satisfies $|\varsigma|=1$.
The multiplication in $\GJ$ is then given by
\begin{gather}\label{eqn:prp-jac:mltGJ}
(\gamma,(\lambda,\mu),\varsigma)(\gamma',(\lambda',\mu'),\varsigma')
=
\left(\gamma\gamma',(\lambda,\mu)\gamma'+(\lambda',\mu'),\varsigma\varsigma'\ex\left(\det\left(\begin{matrix}(\lambda,\mu)\gamma'\\(\lambda',\mu')\end{matrix}\right)\right)\right).
\end{gather}
So in particular, the subgroup $\{(I,(0,0),\varsigma)\}=S^1$ is central. 
We write $\barGJ$ for the quotient of $\GJ$ by this copy of $S^1$. 

Next, 
for $N$ a positive integer, let 
$\GammaOJ(N)$ denote the subgroup of $G^{\rm J}$ composed of the triples 
$(\gamma,(\lambda,\mu),1)$ with $\gamma\in \Gamma_0(N)$ 
and $(\lambda,\mu)\in\ZZ^2$. 
We tacitly identify $\GammaOJ(N)$ with its image in $\barGJ=\SL_2(\RR)\ltimes\RR^2$ under the natural projection when convenient, and let $\GammaooJ$ denote the intersection of the groups $\GammaOJ(N)$, so that
\begin{gather}\label{eqn:prp-jac:GammaooJ}
	\GammaooJ=
	\left\{
	\left.\left(\left(\begin{matrix}1&n\\0&1\end{matrix}\right),(\lambda,\mu),1\right)
	\right| n,\lambda,\mu\in\ZZ
	\right\}.
\end{gather}

Finally we say that $\Gamma<\GJ$ has {\em level $N$} if the image of $\Gamma$ in $\barGJ$ under the natural projection contains $\GammaOJ(N)$ to finite index for some $N$, and if $N$ is the minimal positive integer for which this is true. We also assume that $\Gamma$ contains $\GammaooJ=\ZZ\ltimes\ZZ^2$ (\ref{eqn:prp-jac:GammaooJ}). A subgroup $\Gamma<\GJ$ with {\em level} is a subgroup with level $N$ for some $N$.

We will require a notion of cusp for subgroups of $\GJ$ with level in the above sense. 
To formulate this we 
consider the usual
transitive action of $\SL_2(\RR)$ on the real projective line $\PP^1(\RR)=\RR\cup\{\infty\}$, and
obtain from this an action of $\GJ$ on $\PP^1(\RR)$ by letting the normal subgroup $\{(I,(\lambda,\mu),\varsigma)\}=\RR^2\cdot S^1$ 
(cf.\ (\ref{eqn:prp-jac:mltGJ}))
act trivially. 
Thus an arbitrary subgroup $\Gamma<\GJ$ acts naturally on $\PP^1(\RR)$ by restriction. 
If $\Gamma=\GammaOJ(N)$ for some $N$ then this action preserves the subset $\PP^1(\QQ)=\QQ\cup\{\infty\}$ of rational points,  
and by our definitions, this statement holds true also for $\Gamma<\GJ$ a subgroup with level (cf.\ e.g.\ Proposition 1.10.2 of \cite{Dun_ArthGrpsAffE8Dyn}).
So given $\Gamma<\GJ$ with level 
we may consider the set 
\begin{gather}\label{eqn:prp-jac:cuspsofGamma}
	\Gamma\backslash\PP^1(\QQ) = \{\Gamma\cdot\a\mid \a\in\QQ\cup\{\infty\}\}
\end{gather}
of
orbits of $\Gamma$ on $\PP^1(\QQ)$. We call these orbits (\ref{eqn:prp-jac:cuspsofGamma}) the {\em cusps} of $\Gamma$, and we refer to the orbit $\Gamma\cdot\infty$ containing $\infty$ as the {\em infinite cusp} of $\Gamma$.

For $\Gamma$ a subgroup of $\GJ$ with level, 
a {\em character} of $\Gamma$ will mean a morphism of groups $\rho:\Gamma\to \CC^*$ that is trivial when restricted to $\GammaooJ$ (\ref{eqn:prp-jac:GammaooJ}).
For such $\Gamma$ and $\rho$, and for integers $k$ and $m$, we now define the {\em weight} $k$, {\em index} $m$ and {\em character} $\rho$ action of $\Gamma$ on a smooth function $\phi:\HH\times \CC\to\CC$ by setting
\begin{gather}\label{eqn:prp-jac:kmrhoaction}
	(\phi|_{k,m,\rho}(\gamma,(\lambda,\mu),\varsigma))(\tau,z)
	:=
	(\phi|_{k,m}(\gamma,(\lambda,\mu),\varsigma))(\tau,z)
	\rho(\gamma,(\lambda,\mu),\varsigma)
\end{gather}
for $(\gamma,(\lambda,\mu),\varsigma)\in \Gamma$, where $\phi\mapsto \phi|_{k,m}(\gamma,(\lambda,\mu),\varsigma)$ denotes the usual weight $k$ and index $m$ action of $\GJ$ (as in e.g.\ Theorem 1.4 in \cite{MR781735}).

To formulate the notion of Jacobi form for a group $\Gamma<\GJ$ as above say that a smooth function $\phi:\HH\times \CC\to \CC$ is {\em elliptic} of index $m$ for $\Gamma$ with character $\rho$ if it is invariant for the restriction of the action (\ref{eqn:prp-jac:kmrhoaction}) to the intersection of $\Gamma$ with the normal subgroup $\RR^2\cdot S^1$ of $\GJ$. 
This restricted action is 
given explicitly by
\begin{gather}
	(\phi|_{k,m,\rho}(I,(\lambda,\mu),\varsigma))(\tau,z)
	:=
	\ex\left(m(\lambda^2\tau+2\lambda z+\lambda\mu+\varsigma)\right)
	\phi(\tau,z+\lambda\tau+\mu)
	\rho(I,(\lambda,\mu),\varsigma).	
\end{gather}
In particular, it is independent of $k$.

Given our 
assumptions on $\rho$, an elliptic function of index $m$ for $\Gamma$ with character $\rho$
admits a {\em theta-decomposition}
\begin{gather}\label{eqn:prp-jac:thetadecomp}
	\phi(\tau,z)
	= \sum_{r\xmod 2m} h_r(\tau)\theta_{m,r}(\tau,z),
\end{gather}
where the {\em theta-coefficients} $h_r$ are smooth functions on $\HH$, and
where the theta series $\theta_{m,r}$ 
are defined for integers $m$ and $r$, with $m$ positive, by setting
\begin{gather}\label{eqn:prp-jac:thetamr}
	\theta_{m,r}(\tau,z):=\sum_{s\equiv r\xmod 2m}q^{\frac{s^2}{4m}}y^s. 
\end{gather}

Suppose now that $\phi$ is a holomorphic function on $\HH\times \CC$ that is elliptic of index $m$ for $\Gamma$ with character $\rho$. 
In this work we say that $\phi$ is a {\em mock Jacobi form} of weight $k$ and index $m$ for $\Gamma$ with character $\rho$ if there exist holomorphic functions $g_r:\HH\to \CC$, for $r\xmod 2m$, such that if we set 
\begin{gather}\label{eqn:prp-jac:hathr}
	\widehat h_r(\tau):=h_r(\tau) + g_r^*(\tau),
\end{gather}
where $g_r^*$ is the weight $k-\frac12$ Eichler integral of $g_r$ (see e.g.\ (7.2) of \cite{Dabholkar:2012nd}),
then the 
real-analytic function $\widehat \phi$, given by
\begin{gather}\label{eqn:prp-jac:hatvarphi}
	\widehat\phi(\tau,z):=\sum_{r\xmod 2m}\widehat h_r(\tau)\theta_{m,r}(\tau,z),
\end{gather}
is invariant for the weight $k$, index $m$ and character $\rho$ action (\ref{eqn:prp-jac:kmrhoaction}) of $\Gamma$. 
For completeness we note that if 
\begin{gather}
g_r(\tau)=\sum_{D\geq 0}B(D,r)q^{\frac{D}{4m}}
\end{gather}
is the Fourier expansion of $g_r$, then $g_r^*$ in (\ref{eqn:prp-jac:hathr}) may be expressed explicitly as
\begin{gather}
	g_r^*(\tau)=\overline{B(0,r)}
	\frac{(4\pi \Im(\tau))^{-k+\frac32}}{k-\frac32}
	+\sum_{D>0}\left(\frac{D}{4m}\right)^{k-\frac32}
	\overline{B(D,r)}\Gamma(\tfrac32-k,4\pi\tfrac{D}{4m}\Im(\tau))q^{-\frac{D}{4m}},
\end{gather}
where $\Gamma(\frac32-k,x):=\int_x^\infty t^{-k+\frac12}e^{-t}{\rm d}t$ is the incomplete gamma function.

For $\phi$ a mock Jacobi form, with theta-coefficients $h_r$ as in (\ref{eqn:prp-jac:thetadecomp}), 
we call 
$\widehat h_r$ as in (\ref{eqn:prp-jac:hathr}) the {\em completion} of $h_r$,
and call 
$\widehat \phi$ as in (\ref{eqn:prp-jac:hatvarphi}) the {\em completion} of $\phi$. 
Under our hypotheses on $\rho$ the theta-coefficients $h_r$ of a mock Jacobi form for $\Gamma$ with character $\rho$
admit {\em Fourier series} expansions of the form
\begin{gather}\label{eqn:prp-jac:hr}
	h_r(\tau)=\sum_{D\equiv r^2\xmod 4m} C_\phi(D,r)q^{-\frac{D}{4m}}.
\end{gather}

We write the theta-decomposition (\ref{eqn:prp-jac:thetadecomp}) compactly as 
$\phi(\tau,z) = h(\tau)^{\rm t}\theta_m(\tau,z)$
or $\phi=h^{\rm t}\theta_m$
when convenient, taking $h=(h_r)$ to be the vector-valued function with the theta-coefficients $h_r$ (\ref{eqn:prp-jac:hr}) as its components, 
taking $\theta_m=(\theta_{m,r})$ to be the vector-valued function whose components are the theta series $\theta_{m,r}$ (\ref{eqn:prp-jac:thetamr}),
and using the superscript in $h(\tau)^{\rm t}$ and $h^{\rm t}$ to denote matrix transposition.

We write  $\mpt(\ZZ)$ for the metaplectic double cover of the modular group, and realize it as the set of pairs $(\gamma,\upsilon)$, where $\gamma\in \SL_2(\ZZ)$ and $\upsilon:\HH\to\CC$ is either of the two holomorphic functions such that $\upsilon(\tau)^2=c\tau+d$ when $(c,d)$ is the lower row of $\gamma$. 
We also recall the {\em index $m$} Weil representation $\varrho_m:\mpt(\ZZ)\to \GL_{2m}(\CC)$, defined by requiring that
\begin{gather}\label{eqn:prp-jac:barvarrhomtheta}
	\overline{\varrho_m(\gamma,\upsilon)}
	\th_m\left(\frac{a\tau+b}{c\tau+d},\frac{z}{c\tau+d}\right)\frac1{\upsilon(\tau)}\ex\left(-\frac{cmz^2}{c\tau+d}\right)
	=\theta_m(\tau,z)
\end{gather}
when $\gamma=\left(\begin{smallmatrix}a&b\\c&d\end{smallmatrix}\right)$, for $(\gamma,\upsilon) \in \mpt(\ZZ)$, 
and use this to define {weakly holomorphic}, 
{holomorphic}, and {cuspidal} mock Jacobi forms as follows. 
For $\phi=h^{\rm t}\theta_m$ a mock Jacobi form of weight $k$ and index $m$ for $\Gamma$ with character $\rho$ we say that $\phi$ is {\em weakly holomorphic} if there exists a constant $C>0$ such that 
\begin{gather}\label{eqn:prp-jac:hath_exp}
	{\varrho_m(\gamma,\upsilon)} (\widehat h |_{k-\frac12} (\gamma,\upsilon))(\tau) = O(e^{C\Im(\tau)})
\end{gather}
as $\Im(\tau)\to \infty$, for all $(\gamma,\upsilon)\in \mpt(\ZZ)$, 
where $\widehat h = (\widehat h_r)$, and
where $\widehat h\mapsto \widehat h|_{k-\frac12}(\gamma,\upsilon)$ denotes the usual weight $k-\frac12$ action of the metaplectic group on vector-valued functions on $\HH$ (see e.g.\ \S~3.2 of \cite{cncsga}).
We say that $\phi$ is {\em holomorphic} if the exponential bound (\ref{eqn:prp-jac:hath_exp}) can be replaced with boundedness, 
\begin{gather}\label{eqn:prp-jac:hath_bounded}
	{\varrho_m(\gamma,\upsilon)} (\widehat h |_{k-\frac12} (\gamma,\upsilon))(\tau) = O(1)
\end{gather}
as $\Im(\tau)\to \infty$, for all $(\gamma,\upsilon)\in \mpt(\ZZ)$, and say that $\phi$ is {\em cuspidal} if boundedness (\ref{eqn:prp-jac:hath_bounded}) can in fact be replaced with vanishing,
\begin{gather}\label{eqn:prp-jac:hath_vanishing}
	{\varrho_m(\gamma,\upsilon)} (\widehat h |_{k-\frac12} (\gamma,\upsilon))(\tau) \to 0
\end{gather}
as $\Im(\tau)\to \infty$, for all $(\gamma,\upsilon)\in \mpt(\ZZ)$. 
Note that a cuspidal mock Jacobi form necessarily has vanishing shadow, so there are no cuspidal mock Jacobi forms that are not actually Jacobi forms. This can be seen by using the pairing introduced by Bruinier--Funke in Proposition 3.5 of \cite{MR2097357}. (See Proposition 3.2.1 of \cite{MR4127159} for a formulation of the Bruinier--Funke pairing in terms of Jacobi forms.)

With the above definitions in place we let
$\JJ^\wh_{k,m}(\Gamma,\rho)$ 
denote the space of weakly holomorphic (\ref{eqn:prp-jac:hath_exp}) mock Jacobi forms of weight $k$ and index 
$m$ 
for $\Gamma$ with character $\rho$, and write $\JJ_{k,m}(\Gamma,\rho)$  for the subspace of mock Jacobi forms that are holomorphic (\ref{eqn:prp-jac:hath_bounded}). 
The {\em holomorphic Jacobi forms} of weight $k$ and index $m$ for $\Gamma$ with character $\rho$ are the forms $\phi=h^{\rm t}\theta_m$ in $\JJ_{k,m}(\Gamma,\rho)$ for which the $g^*_r$ in (\ref{eqn:prp-jac:hathr}) all vanish, so that $\widehat h=h$, \&c. 
We denote the space they comprise by $J_{k,m}(\Gamma,\rho)$, and write $S_{k,m}(\Gamma,\rho)$ for the subspace comprised of 
Jacobi forms that are cuspidal.
We suppress the character $\rho$ from these notations when it is trivial. 

In the remainder of this section we offer
a quick review of
some of 
the 
meromorphic Jacobi 
form theory that
appears in
\S~8 of \cite{Dabholkar:2012nd}. 
Briefly, a meromorphic Jacobi form of weight $k$ and index $m$ 
for $\Gamma$ with character $\rho$ 
is a function that 
is invariant under the corresponding action (\ref{eqn:prp-jac:kmrhoaction}) of $\Gamma$,  
but is allowed to have poles in the $z$ variable.
A classic example is the Weierstrass elliptic function,
\begin{gather}\label{eqn:prp-jac:wp}
	\wp(\tau,z) := \frac{1}{z^2}+ 
	\sum_{\substack{(a,b)\in\ZZ^2\\(a,b)\neq (0,0)}} 
	\frac{1}{(a\tau+b-z)^2}-\frac{1}{(a\tau+b)^2}
	,
\end{gather}
which has weight $2$ and index $0$ for $\GammaOJ(1)=\SL_2(\ZZ)\ltimes \ZZ^2$ (with trivial character), and a double pole at $z=a\tau+b$, for all $a,b\in \ZZ$.
An example of particular relevance in this work is
\begin{gather}\label{eqn:prp-jac:psi1}
	\psi^{(1)}(\tau,z):=\frac{1}{12}\left(\frac{9}{4\pi^2}\wp(\tau,z)^2-E_4(\tau)\right)\frac{\vartheta_1(\tau,z)^2}{\eta(\tau)^6},
\end{gather}
where $E_4$ is the weight $4$ Eisenstein series for $\SL_2(\ZZ)$, normalized so that $E_4(\tau)=1+O(q)$ as $\Im(\tau)\to \infty$, and where the {\em Jacobi theta function}  and {\em Dedekind eta function} are defined by setting
\begin{gather}
	\vartheta_1(\tau,z):=-\theta_{1,1}(\tfrac12\tau,\tfrac12(z+\tfrac12)),\label{eqn:prp-jac:vartheta1}
	\\
	\eta(\tau):=q^{\frac{1}{24}}\prod_{n>0}(1-q^n),\label{eqn:prp-jac:eta}
\end{gather}
respectively, for $\tau\in \HH$ and $z\in \CC$, where $\theta_{1,1}$ is as in (\ref{eqn:prp-jac:thetamr}).
 Similar to $\wp$, 
 the function $\psi^{(1)}$ 
 has a double pole at every lattice point $z\in \ZZ\tau+\ZZ$, but in contrast to $\wp$ it transforms with weight $2$ and index $1$ under the action of $\GammaOJ(1)$, 
 rather than weight $2$ and index $0$.

Observe now that, because of its poles, 
a meromorphic Jacobi form generally does not admit a theta-decomposition as in (\ref{eqn:prp-jac:thetadecomp}).
A solution to this problem is proposed in \cite{Dabholkar:2012nd},
following earlier work \cite{Zwegers_thesis} of Zwegers.
Under suitable conditions on the meromorphic Jacobi form $\psi$, the authors of \cite{Dabholkar:2012nd} introduce a canonical 
decomposition 
\begin{gather}\label{eqn:prp-jac:polarfinite}
\psi=\psi^P+\psi^F,
\end{gather} 
where $\psi^P$ and $\psi^F$ are called the {\em polar} and {\em finite} parts of $\psi$, respectively. 
Here 
the 
polar part $\psi^P$ 
may be realized as an average over the local behavior of $\psi(\tau,z)$ as $z$ approaches its poles (within a fundamental domain for $\ZZ\tau+\ZZ$), 
and 
the finite part $\psi^F$ is defined by setting
\begin{gather}\label{eqn:prp-jac:psiF}
	\psi^F(\tau,z) = \sum_{r\xmod 2m} h_r(\tau)\theta_{m,r}(\tau,z)
\end{gather}
(cf.\ (\ref{eqn:prp-jac:thetadecomp})), where the $h_r(\tau)$ are computed by imposing certain choices on the (otherwise ambiguous) integral expression
\begin{gather}\label{eqn:prp-jac:hrint}
h_r(\tau)q^{\frac{r^2}{4m}}=\int_{z_0}^{z_0+1}\psi(\tau,z)\ex(-rz){\rm d}z.
\end{gather}
In particular, the finite part of a meromorphic Jacobi form admits a theta-decomposition (\ref{eqn:prp-jac:psiF}) by construction.

Note that (\ref{eqn:prp-jac:hrint}) unambiguously recovers the theta-coefficients of $\psi$, for any choice of $z_0\in\CC$, and any path from $z_0$ to $z_0+1$, if $\psi$ is holomorphic in $z$ and elliptic of index $m$ (cf.\ (\ref{eqn:prp-jac:thetadecomp})). We refer to \S~8 of \cite{Dabholkar:2012nd} for details on how to interpret (\ref{eqn:prp-jac:hrint}) in the meromorphic case.

We conclude by detailing the decomposition (\ref{eqn:prp-jac:polarfinite}) in the special case that $\psi=\psi^{(1)}$ is as in (\ref{eqn:prp-jac:psi1}). 
For this we follow \cite{Dabholkar:2012nd} in defining the {\em averaging operator} of index $m$, by setting
\begin{gather}\label{eqn:prp-jac:Avm}
\Av^{(m)}(F(y)):=\sum_{s\in\ZZ}y^{2ms}q^{ms^2}F(yq^s)
\end{gather} 
for $m$ a non-negative integer and $F(y)$ a function of polynomial growth. 
We are then able to identify the polar part $\psi^{(1),P}$ of $\psi^{(1)}$ 
as
\begin{gather}\label{eqn:prp-jac:psi1P}
\psi^{(1),P}(\tau,z)=-12\Av^{(1)}({y}{(1-y)^{-2}}),
\end{gather}
and from the discussion of Example 5 in \S~8.5 of \cite{Dabholkar:2012nd} we obtain that the finite part $\psi^{(1),F}$ satisfies
\begin{gather}\label{eqn:prp-jac:psi1F}
	\psi^{(1),F}(\tau,z)
	=24\sHGH(\tau,z),
\end{gather}
where $\sHGH=\sHGH_1$ is the Hurwitz class number generating function, as defined in (\ref{eqn:prp-cln:msHN}).

\subsection{Optimality}
\label{sec:prp-opt}

In this section we explain optimality for holomorphic mock Jacobi forms of arbitrary integer weight and positive integer index. The exposition is similar to that of \S~3.3 of \cite{cncsga} except that here we allow for forms that are automorphic for more general subgroups of the Jacobi group, as described in \S~3.2, and also consider characters.

To begin 
we 
fix an integer $k$ and a positive integer $m$. 
Then for $\Gamma$ and $\rho$ as in (\ref{eqn:prp-jac:kmrhoaction})
we say that
$\phi\in\JJ_{k,m}(\Gamma,\rho)$ 
is {\em optimal} if 
\begin{gather}\label{eqn:prp-opt:phi_optimal}
	\left({\varrho_m(\gamma,\upsilon)} \widehat h |_{k-\frac12} (\gamma,\upsilon)\right)(\tau) \to 0
\end{gather}
as $\Im(\tau)\to \infty$ (cf.\ (\ref{eqn:prp-jac:hath_vanishing})), whenever $(\gamma,\upsilon)\in\mpt(\ZZ)$ 
is such that $\gamma\cdot\infty$ does not belong to the infinite cusp of $\Gamma$ (cf.\ (\ref{eqn:prp-jac:cuspsofGamma})). 
We then refine this notion by saying that
$\phi\in\JJ_{k,m}(\Gamma,\rho)$ 
is {\em $\opc$-optimal}, for a given constant $\opc$, if $\phi$ is optimal (\ref{eqn:prp-opt:phi_optimal}) and satisfies
\begin{gather}\label{eqn:prp-opt:phiequalsminusc}
	\phi(\tau,z)=-\opc+O(q),
\end{gather}
for every $z\in\CC$, as $\Im(\tau)\to\infty$. 

As in \cite{cncsga}
we are 
interested in situations whereby collections of optimal mock Jacobi forms are organized by finite groups. 
To 
setup for this we 
let $\sG$ be a finite group, and consider an assignment
\begin{gather}\label{eqn:prp-opt:alpha}
\ua:[\sg]\mapsto \ua(\sg)=(\Gamma_\sg,\rho_\sg)
\end{gather}
of pairs 
to the conjugacy classes $[\sg]$ 
of 
$\sG$, where for each $\sg\in \sG$ the first component $\Gamma_\sg$ of $\ua(\sg)$ is a subgroup of $\GJ$ with level $N=o(\sg)$ in the sense of \S\ref{sec:prp-jac}, 
and the second 
component 
$\rho_\sg$ 
is 
a character of $\Gamma_\sg$ as in (\ref{eqn:prp-jac:kmrhoaction}).
Also, given a 
virtual graded $\sG$-module $W$ 
of the form
\begin{gather}\label{eqn:prp-opt:W}
W=\bigoplus_{r\xmod 2m}\bigoplus_{D\equiv r^2\xmod 4m}W_{r,\frac{D}{4m}}
\end{gather} 
we define 
the associated {\em McKay--Thompson series} $\phi^W_\sg$, for each $\sg\in \sG$,  
by 
requiring that 
\begin{gather}\label{eqn:prp-opt:phiWsg}
\phi^W_\sg(\tau,z)=\sum_{r\xmod 2m}h^W_{\sg,r}(\tau)\theta_{m,r}(\tau,z)
\end{gather} 
(cf.\ (\ref{eqn:prp-jac:hatvarphi})), 
where  
$h^W_{\sg,r}$ is
defined, 
for each integer $r$ modulo $2m$,
by 
setting
\begin{gather}\label{eqn:prp-opt:hWsgr}
	h^W_{\sg,r}(\tau):=\sum_{D\equiv r^2\xmod 4m} 
	\tr\left(g|W_{r,\frac{D}{4m}}\right)q^{-\frac{D}{4m}}.
\end{gather}
As in \cite{cncsga} we use the term virtual $\sG$-module to refer to an integer combination of irreducible ordinary characters of $\sG$, and use the term virtual graded $\sG$-module to refer to an indexed collection of such things. See \S~3.3 of op.\ cit.\ for more detail.

With this setup in place 
we say that $W$ as in (\ref{eqn:prp-opt:W}) is {\em $\opc$-optimal (mock Jacobi)} of weight $k$ and index $m$ for $(\sG,\ua)$ if 
\begin{gather}\label{eqn:prp-opt:alpha_condition}
\phi^W_\sg \in
\JJ_{k,m}(\Gamma_\sg,\rho_\sg),
\end{gather}
and if $\phi^W_\sg$ also satisfies the $\opc$-optimality conditions 
(\ref{eqn:prp-opt:phi_optimal}-\ref{eqn:prp-opt:phiequalsminusc}),
for every $\sg\in\sG$, 
and we say that
$W$ as in (\ref{eqn:prp-opt:W}) {\em optimal} for $(\sG,\ua)$ if it is $\opc$-optimal for some $\opc$.

We are interested in classifying the 
optimal 
virtual graded modules of a given weight and index for a pair $(\sG,\ua)$, where $\sG$ is a finite group and $\ua$ is as in (\ref{eqn:prp-opt:alpha}).
With this in mind we let
$\mc{W}^\opt_{k,m}(\sG,\ua)$ denote the set of optimal virtual graded modules of weight $k$ and index $m$ for $(\sG,\ua)$, and 
for each $\opc\in \ZZ$ let $\mc{W}^\opt_{k,m}(\sG,\ua)_\opc$ denote the subset of modules that are $\opc$-optimal.
As in \cite{cncsga} we note that $\mc{W}^\opt_{k,m}(\sG,\ua)$ and $\mc{W}^\opt_{k,m}(\sG,\ua)_0$ are naturally free abelian groups (cf.\ (3.3.11) of op.\ cit.), and 
\begin{gather}\label{eqn:prp-opt:mcWoptkmsGua}
	\mc{W}^\opt_{k,m}(\sG,\ua)=\sum_{\opc\in\ZZ}\mc{W}^\opt_{k,m}(\sG,\ua)_\opc
\end{gather}
is a decomposition of the former into modules for the latter. We also write $\mc{L}^\opt_{k,m}(\sG,\ua)$ for the lattice structure on $\mc{W}^\opt_{k,m}(\sG,\ua)_0$ that we obtain by 
embedding it in 
\begin{gather}\label{eqn:prp-opt:LkmsG}
	L_{k,m}(\sG) := S_{k,m}(\#\sG)_\ZZ\otimes_\ZZ R(\sG)
\end{gather}
in the natural way. 
(Here $R(\sG)$ denotes the Grothendieck group of the category of finitely generated $\sG$-modules, which we identify with the free $\ZZ$-module generated by the irreducible ordinary characters of $\sG$. We refer to (3.3.12-3.3.14) of \cite{cncsga} for the lattice structure on $L_{k,m}(\sG)$, 
and refer to (3.3.15-3.3.19) of op.\ cit.\ for the map that identifies $\mc{W}^\opt_{k,m}(\sG,\ua)_0$ as a sublattice.)

Define 
$\opc^\opt_{k,m}(\sG,\ua)$ to be the minimal positive integer for which a $\opc$-optimal virtual graded module of weight $k$ and index $m$ for $(\sG,\ua)$ exists, 
\begin{gather}\label{eqn:prp-opt:opcopt}
	\opc^\opt_{k,m}(\sG,\ua) := \min\left\{\opc\in \ZZ^+\mid \mc{W}^\opt_{k,m}(\sG,\ua)_\opc\neq \emptyset\right\}.
\end{gather}
Then we have the following counterpart to Proposition 3.3.1 of \cite{cncsga}, by exactly the same argument.
\begin{pro}\label{pro:prp-opt:mcWropmcLrop}
The sets $\mc{W}^\opt_{k,m}(\sG,\ua)$ and $\mc{L}^\opt_{k,m}(\sG,\ua)$ are naturally free abelian groups of finite rank. 
If $\opc\equiv 0 \xmod \opc^\opt_{k,m}(\sG,\ua)$ then $\mc{W}^\opt_{k,m}(\sG,\ua)_\opc$ is naturally a $\mc{L}^\opt_{k,m}(\sG,\ua)$-torsor. 
If $\opc\not\equiv 0 \xmod \opc^\opt_{k,m}(\sG,\ua)$ then $\mc{W}^\opt_{k,m}(\sG,\ua)_\opc$ is empty.
\end{pro}

Motivated by 
Proposition \ref{pro:prp-opt:mcWropmcLrop} 
we refer to the determination of $\opc^\opt_{k,m}(\sG,\ua)$ and $\mc{L}^\opt_{k,m}(\sG,\ua)$, for a fixed choice of $\sG$ and $\ua$, as the {\em classification problem} for optimal virtual graded modules of weight $k$ and index $m$ for $(\sG,\ua)$.

To conclude this section we define the {\em default} choice for $\ua$ as in (\ref{eqn:prp-opt:alpha}) to be the assignment $\ua(\sg)=(\Gamma_\sg,\rho_\sg)$ for which 
\begin{gather}\label{eqn:prp-opt:default_alpha}
\Gamma_\sg=\GammaOJ(o(\sg))
\end{gather}
and $\rho_\sg$ is trivial, for each $\sg\in \sG$. 
We indicate this default choice of $\ua$ by suppressing it from notation, writing $\mc{W}^\opt_{k,m}(\sG)$ for $\mc{W}^\opt_{k,m}(\sG,\ua)$, 
and writing $\mc{L}^\opt_{k,m}(\sG)$ for $\mc{L}^\opt_{k,m}(\sG,\ua)$, \&c., when $\ua$ is as in (\ref{eqn:prp-opt:default_alpha}).

\subsection{Our Main Focus}\label{sec:prp-fcs}

Having discussed mock Jacobi forms with arbitrary integer weight and positive integer index in \S\S~\ref{sec:prp-jac}-\ref{sec:prp-opt}, we 
now review some of the special features that arise 
when the weight is $2$ and the index is $1$, since this is the situation upon which we mostly focus in the remainder of this work.

To begin we note that for a mock Jacobi form $\phi$ with index $m=1$ we may unambiguously write $C_\phi(D)$ for the Fourier coefficient $C_\phi(D,r)$ (cf.\ (\ref{eqn:prp-jac:hr})), because the latter depends only on $D$ and the parity of $r$ (cf.\ (\ref{eqn:prp-jac:thetadecomp})), and the relation $D\equiv r^2\xmod 4$ forces the parity of $D$ and $r$ to match.
With this convention we have a Fourier expansion
\begin{gather}\label{eqn:prp-fcs:phi}
	\phi(\tau,z)=\sum_{n,s\in\ZZ}C_\phi(s^2-4n)q^ny^s
\end{gather} 
when $\phi$ has index $1$ (cf.\ (\ref{eqn:prp-cln:msHN}), (\ref{eqn:prp-cln:msHCohN})). 
(More generally we may drop the $r$ from $C_\phi(D,r)$ whenever the index $m$ of $\phi$ is not composite, for in this case $r^2\equiv s^2\xmod 4m$ implies $r\equiv \pm s\xmod 2m$.)

By a similar token, for
$W$ as in (\ref{eqn:prp-opt:W}) with $m=1$, 
there is no loss in considering $\check{W}=\bigoplus_D \check{W}_D$ in place of $W$,  
where $\check{W}_D:=W_{D,\frac{D}{4}}$. 
As in \cite{cncsga} we adopt this simplification, and also drop the accents from $\check{W}$ and $\check{W}_D$ henceforth.
Thus, from now on
we write simply
\begin{gather}\label{eqn:prp-fcs:W}
	W=\bigoplus_D W_D
\end{gather}
for the grading of a virtual graded $\sG$-module $W$ as in (\ref{eqn:prp-opt:W}) when $m=1$, where each $W_D$ in (\ref{eqn:prp-fcs:W}) is taken to be $W_{D,\frac{D}{4}}$ in (\ref{eqn:prp-opt:W}).
With this convention the 
McKay--Thompson series $\phi^W_\sg$ 
associated to $W$
(see (\ref{eqn:prp-opt:phiWsg}-\ref{eqn:prp-opt:hWsgr})) may be defined by setting
\begin{gather}\label{eqn:prp-fcs:phiWsg}
	\phi^W_\sg(\tau,z):=\sum_{n,s\in\ZZ}\tr(\sg|W_{s^2-4n})q^ny^s
\end{gather}
(cf.\ (\ref{eqn:prp-fcs:phi})).

The main significance for us of the weight $k=2$ is that it plays host to the generalized Hurwitz class number generating functions $\sHGH_N$ of (\ref{eqn:prp-cln:msHN}), and the Cohen--Eisenstein series $\sHCE_N$ of (\ref{eqn:prp-cln:msHCohN}). 
We put this precisely as follows.
\begin{pro}\label{pro:prp-fcs:msHHurNmockmsHCohNmod}
For $N$ a positive integer the function $\sHGH_N$ 
belongs to $\JJ_{2,1}(\GammaOJ(N))$. 
For $N$ prime the function $\sHCE_N$ belongs to $J_{2,1}(\GammaOJ(N))$.
\end{pro}
See Proposition 3.4.1 in \cite{cncsga} for the proof of Proposition \ref{pro:prp-fcs:msHHurNmockmsHCohNmod} above.

To construct optimal mock Jacobi forms we take a positive integer $N$ and set
\begin{gather}\label{eqn:prp-fcs:msRN}
	\sHR_N
	:=\frac{12}{\phi(N)}\sum_{M|N}\mu\left(\frac NM\right)\frac{M}{\indo(M)}\sHGH_M,
\end{gather}
where $\phi(N)$ denotes the Euler totient function, $\mu(N)$ is the M\"obius function, and $\indo(N)$ is as in (\ref{eqn:prp-cln:HN0}). Then, as explained in \S~3.4 of \cite{cncsga}, the $\sHR_N$ are optimal.
\begin{pro}\label{pro:prp-fcs:msRN_1opt}
For any positive integer $N$ the function $\sHR_N$ is a $1$-optimal element of $\JJ_{2,1}(\GammaOJ(N))$. 
\end{pro}

Note that the Fourier coefficients of $\sHR_N$ are generally not integers. For example, $\opc=2$ is the smallest positive integer for which $\opc\sHR_5$ has integer coefficients. (Cf.\ the proof of Theorem \ref{thm:mat-opt:tor}.) Also, $\JJ_{2,1}(\GammaOJ(1))$ is spanned by $\sHR_1=12\sHGH$ according to Proposition 3.4.4 of \cite{cncsga}. 
Thus for any finite group $\sG$, and for any choice of $\ua$ in (\ref{eqn:prp-opt:alpha}) such that $\Gamma_\se = \GammaOJ(1)$ and $\rho_{\se}$ is trivial, 
a $\opc$-optimal virtual graded module $W=\bigoplus_D W_D$ of weight $2$ and index $1$ for $(\sG,\ua)$ (cf.\ (\ref{eqn:prp-fcs:W})) satisfies
\begin{gather}
	\phi^W_\se(\tau,z) = \sum_{n,s\in\ZZ} \dim(W_{s^2-4n})q^ny^s = \opc\sHR_1(\tau,z) = 12\opc\sHGH(\tau,z)
\end{gather}
(cf.\ (\ref{eqn:prp-fcs:phiWsg})), where $\se$ denotes the identity element of $\sG$.

\section{Results}\label{sec:mat}

In this section we present concrete examples of the 
classification problem formulated in \S~\ref{sec:prp-opt} (cf.\ Proposition \ref{pro:prp-opt:mcWropmcLrop}), in the special setting described in \S~\ref{sec:prp-fcs}, and also explore some arithmetic consequences.
Specifically, we 
first explain the connection between class numbers and the 
$\ell=3$ case of umbral moonshine in \S~\ref{sec:mat-um3}. Then we use this as a springboard to
classify optimal modules for the smallest sporadic Mathieu group $\sMa$, for two different assignments of characters, in \S~\ref{sec:mat-opt}. 
Finally, in \S~\ref{sec:mat-art} we demonstrate a connection between the $\sMa$-modules of \S~\ref{sec:mat-opt} and the congruent number problem from antiquity.

\subsection{Umbral Moonshine}\label{sec:mat-um3}

In this section we 
use the proof of the $\ell=3$ case of umbral moonshine (see \cite{UM,MUM,MR3766220} and \cite{MR3433373}) to establish
the existence of a $2$-optimal module for 
the unique non-trivial double cover $2.\sMb$ of the Mathieu group $\sMb$ (cf.\ \cite{atlas}), and a particular assignment $\ua:[\sg]\to (\Gamma_\sg,\rho_\sg)$ (cf.\ (\ref{eqn:prp-opt:alpha})). 

To get started let us first formulate the relevant function from umbral moonshine. 
For this we
recall
the Jacobi theta function $\vartheta_1$ (\ref{eqn:prp-jac:vartheta1}) and Dedekind eta function $\eta$ (\ref{eqn:prp-jac:eta}) from \S~\ref{sec:prp-jac}, 
and define three more Jacobi theta functions,
$\vartheta_j(\tau,z)$ for $j\in \{2,3,4\}$, 
by setting
\begin{gather}
\begin{split}
	\label{eqn:mat-um3:theta}
\vartheta_2(\tau,z)&:=\theta_{1,1}(\tfrac12\tau,\tfrac12z),\\
\vartheta_3(\tau,z)&:=\theta_{1,0}(\tfrac12\tau,\tfrac12z),\\
\vartheta_4(\tau,z)&:=\theta_{1,0}(\tfrac12\tau,\tfrac12(z+\tfrac12)),\\
\end{split}
\end{gather}
for $\tau\in \HH$ and $z\in \CC$, where $\theta_{1,0}$ and $\theta_{1,1}$ are as in (\ref{eqn:prp-jac:thetamr}). We use these functions to specify a weakly holomorphic Jacobi form $Z^{(3)}$ of weight $0$ and index $2$ for $\GammaOJ(1)$ (cf.\ \S~2.5 of \cite{UM}) by defining
\begin{gather}\label{eqn:mat-um3:Z3}
	Z^{(3)}(\tau,z) := 
	4\left(
	\frac{\vartheta_2(\tau,z)^2}{\vartheta_2(\tau,0)^2}
	\frac{\vartheta_3(\tau,z)^2}{\vartheta_3(\tau,0)^2}
	+
	\frac{\vartheta_3(\tau,z)^2}{\vartheta_3(\tau,0)^2}
	\frac{\vartheta_4(\tau,z)^2}{\vartheta_4(\tau,0)^2}
	+
	\frac{\vartheta_4(\tau,z)^2}{\vartheta_4(\tau,0)^2}
	\frac{\vartheta_2(\tau,z)^2}{\vartheta_2(\tau,0)^2}
	\right).
\end{gather}
Then the 
meromorphic Jacobi form of weight $1$ and index $3$ attached to the $\ell=3$ case of umbral moonshine
(cf.\ \S~4.3 of \cite{MUM}), 
denoted  $\psi^{(3)}$, may be defined 
by the formula
\begin{gather}\label{eqn:mat-um3:psi3}
	\psi^{(3)}(\tau,z):=i\frac{\vartheta_1(\tau,2z)\eta(\tau)^3}{\vartheta_1(\tau,z)^2}Z^{(3)}(\tau,z).
\end{gather}
(We remark that the factor in front of $Z^{(3)}$ in (\ref{eqn:mat-um3:psi3}) is a meromorphic Jacobi form of weight $1$ and index $1$ for $\GammaOJ(1)$, with a simple pole at each lattice point $z\in \ZZ\tau+\ZZ$.)

The connection between (\ref{eqn:mat-um3:psi3}) and umbral moonshine is that the finite part of $\psi^{(3)}$, denoted $\psi^{(3),F}$ (cf.\ (\ref{eqn:prp-jac:polarfinite})), satisfies
\begin{gather}\label{eqn:mat-um3:psi3F}
	\psi^{(3),F}(\tau,z) = 
	\sum_{r\xmod 6} H^{(3)}_r(\tau)\theta_{3,r}(\tau,z),
\end{gather}
where $H^{(3)}=(H^{(3)}_r)$ is the McKay--Thompson series attached to the identity element of the umbral group at $\ell=3$, which is none other than $2.\sMb$.

We now consider the quotient of $Z^{(3)}$ (\ref{eqn:mat-um3:Z3}) by the weakly holomorphic Jacobi form $\vartheta_1^2\eta^{-6}$ (cf.\ (\ref{eqn:prp-jac:vartheta1}-\ref{eqn:prp-jac:eta})), 
the latter having weight $-2$ and index $1$, and a double zero at each lattice point $z\in \ZZ\tau+\ZZ$, and compare to the meromorphic Jacobi form $\psi^{(1)}$ of (\ref{eqn:prp-jac:psi1}). After an elementary manipulation of the expressions involved we arrive at a precise coincidence, 
\begin{gather}\label{eqn:mat-um3:psi1}
	\psi^{(1)}(\tau,z)=-i\frac{\eta(\tau)^3}{\vartheta_1(\tau,2z)}\psi^{(3)}(\tau,z)=\frac{\eta(\tau)^6}{\th_1(\tau,z)^2}Z^{(3)}(\tau,z).
\end{gather}
That is to say, we recover (a rescaling of) the Hurwitz class number generating function $\sHGH=\sHGH_1$ (cf.\ (\ref{eqn:prp-cln:msHN})) by taking the finite part (cf.\ (\ref{eqn:prp-jac:psi1F})) of a suitable multiple (\ref{eqn:mat-um3:psi1}) of the meromorphic Jacobi form (\ref{eqn:mat-um3:psi3}) whose finite part (\ref{eqn:mat-um3:psi3F}) captures umbral moonshine at $\ell=3$.

Equipped with (\ref{eqn:mat-um3:psi3}-\ref{eqn:mat-um3:psi1}) we now attach a meromorphic Jacobi form of weight $2$ and index $1$ (with level and generally with non-trivial character) to each $\sg\in 2.\sMb$ by setting
\begin{gather}\label{eqn:mat-um3:psi1g}
\psi^{(1)}_\sg(\tau,z):=-i\frac{\eta(\tau)^3}{\vartheta_1(\tau,2z)}\psi^{(3)}_\sg(\tau,z),
\end{gather}
where $\psi^{(3)}_\sg(\tau,z)$ is the meromorphic Jacobi form of weight $1$ and index $3$ (with level and generally with non-trivial character) attached to $\sg\in 2.\sMb$ by umbral moonshine. 
We define $\calH^{(2.\sMb,\ua)}_\sg$ for $\sg\in 2.\sMb$ to be the finite part of $\psi^{(1)}_\sg$ in (\ref{eqn:mat-um3:psi1g}), so that
\begin{gather}
	\calH^{(2.\sMb,\ua)}_\sg := \psi^{(1),F}_\sg.
\end{gather}

To make (\ref{eqn:mat-um3:psi1g}) more explicit let $\overline{\chi}_{12}$ be the sum of the first two irreducible characters of $2.\sMb$, according to the ordering in \cite{UM,MUM} (which is the same as the ordering given by \cite{GAP4}, for example), and let $\chi_{12}$ denote the unique faithful irreducible character of $2.\sMb$ of degree $12$.
Then we have
\begin{gather}
\psi^{(3)}_\sg(\tau,z)=-{\chi}_{12}(\sg)\mu^{(3),0}(\tau,z)-\overline\chi_{12}(\sg)\mu^{(3),1}(\tau,z)+\sum_{r\xmod 6}H^{(3)}_{\sg,r}(\tau)\theta_{3,r}(\tau,z)
\end{gather}
where 
$\mu^{(3),k}$
and the $H^{(3)}_{\sg,r}$ are as specified in \S~4.3 of \cite{MUM} (cf.\ also \S~B.3.2 of \cite{umrec}).

Correspondingly, we now define a mock Jacobi form $\ms{H}^{(2.\sMb,\ua)}_\sg$ for each $\sg\in 2.\sMb$ by requiring that
\begin{gather}\label{eqn:mat-um3-msH2M12g}
	\psi^{(1)}_\sg(\tau,z)=-{\chi}_{12}(\sg)\mu^{(1),0}(\tau,z)-\overline\chi_{12}(\sg)\mu^{(1),1}(\tau,z)+\ms{H}^{(2.\sMb,\ua)}_\sg(\tau,z)
\end{gather}
where
$\mu^{(1),k}(\tau,z):=\frac12\left(\Av^{(1)}(\frac{y}{(1-y)^2})+(-1)^k\Av^{(1)}(\frac{-y}{(1+y)^2})\right)$ (cf.\ (\ref{eqn:prp-jac:Avm}-\ref{eqn:prp-jac:psi1P})).

It follows from the construction (\ref{eqn:mat-um3:psi1g}-\ref{eqn:mat-um3-msH2M12g}), and the $\ell=3$ case of the main result of \cite{umrec}, that the $\ms{H}^{(2.\sMb,\ua)}_\sg$ arise as the graded trace functions attached to a virtual graded $2.\sMb$-module. 

\begin{pro}\label{pro:mat-um3:WGalpha2nonempty}
There exists a virtual graded $2.\sMb$-module $W=\bigoplus_D W_D$ such that $\phi^W_\sg=\ms{H}^{(2.\sMb,\ua)}_\sg$ for $\sg\in 2.\sMb$.
\end{pro}

It also follows from (\ref{eqn:mat-um3:psi1g}-\ref{eqn:mat-um3-msH2M12g}) that the $\ms{H}^{(2.\sMb,\ua)}_\sg$ are optimal, and have the same characters as those of the meromorphic Jacobi forms $\psi^{(3)}_\sg$.
We have 
$\ms{H}^{(2.\sMb,\ua)}_{\sg}=-2+O(q)$ as $\Im(\tau)\to \infty$ for all $\sg\in 2.\sMb$ so Proposition \ref{pro:mat-um3:WGalpha2nonempty} implies that $\opc^\opt_{2,1}(\sG,\ua)\leq 2$ for $\sG=2.\sMb$ and $\ua$ as above. On the other hand the inequality $\opc^\opt_{2,1}(\sG,\ua)\geq 2$ holds because 
the function $\ms{H}^{(2.\sMb,\ua)}_\sg$ belongs to $\JJ_{2,1}(\GammaOJ(5))$ in case $o(\sg)=5$, the unique-up-to-scale optimal form in this space is $\sHR_5$, and $\opc=2$ is the smallest positive integer such that $\opc\sHR_5$ has integer coefficients. Thus we have proved the following result. 
\begin{thm}\label{thm:mat-um3:tor}
For $\sG=2.\sMb$  and $\ua$ as above we have $\opc^\opt_{2,1}(\sG,\ua)=2$.
\end{thm}

\subsection{Modules}\label{sec:mat-opt}

Here we consider the smallest sporadic simple group $\sMa$, which has order $7920=2^4.3^2.5.11$ (cf.\ \cite{atlas}). 
Our first objective is the following classification (cf.\ Proposition \ref{pro:prp-opt:mcWropmcLrop}) of optimal modules for $\sMa$ (i.e.\ optimal modules for $(\sMa,\ua)$ with the default choice (\ref{eqn:prp-opt:default_alpha}) of $\ua$). To formulate it let $\varphi_{11}$ denote
the unique element of $S_{2,1}(\GammaOJ(11))$ 
such that
\begin{gather}\label{eqn:mat-opt:varphi11}
\varphi_{11}(\tau,z)=(y^{-1}-1+y)q+O(q^2)
\end{gather}
as $\Im(\tau)\to \infty$,  
let $\chi_i$ denote the $i$-th irreducible ordinary character of $\sMa$ (as displayed in Table \ref{tab:chars:m11}), 
and recall the definition (\ref{eqn:prp-opt:LkmsG}) of the lattice $L_{2,1}(\sG)$.

\begin{thm}\label{thm:mat-opt:tor}
For $\sG=\sMa$ we have $\opc^\opt_{2,1}(\sG)=2$, and $\mc{L}^\opt_{2,1}(\sG)$ is the rank $2$ sublattice of $L_{2,1}(\sG)$ generated by 
\begin{gather}
	\begin{split}\label{eqn:mat-opt:tor}
	&\varphi_{11}\otimes(\chi_1-\chi_2-\chi_3-\chi_4-\chi_6+\chi_9),\\
	&\varphi_{11}\otimes(\chi_1-\chi_2-\chi_3-\chi_4-\chi_7+\chi_9).
	\end{split}
\end{gather}
\end{thm}
\begin{rmk}\label{rmk:mat-opt:mcL21sG}
Let $N_{11}$ denote the norm of $\varphi_{11}$ with respect to the Petersson inner product. 
Then from Theorem \ref{thm:mat-opt:tor} we see that $\mc{L}^\opt_{2,1}(\sG)$ is 
a copy of the lattice structure on $\ZZ^2$ for which 
the 
associated quadratic form is $Q(x_1,x_2)=N_{11}^2(3x_1^2+5x_1x_2+3x_2^2)$.
\end{rmk}

Our main task in proving Theorem \ref{thm:mat-opt:tor} is establishing that a $2$-optimal module for $\sMa$ exists. 
To this end we define a holomorphic mock Jacobi form $\ms{H}^\sMa_\sg$ for each $\sg\in \sMa$ by setting 
\begin{gather}\label{eqn:mat-opt:msHM11g}
\ms{H}^{\sMa}_\sg:=
\begin{cases}
2\sHR_{o(\sg)}&\text{for $o(\sg)<11$,}\\
2\sHR_{11}-\frac{11}{5}\varphi_{11}&\text{for $o(\sg)=11$.}
\end{cases}
\end{gather} 
Here
$\sHR_N$ is as in (\ref{eqn:prp-fcs:msRN}) and 
$\varphi_{11}$ is as in (\ref{eqn:mat-opt:varphi11}).
The first few Fourier coefficients of each of the $\ms{H}^\sMa_\sg$ are given in Table \ref{tab:coeffs:m11}. 

The $\sHR_N$ are optimal for all $N$ according to Proposition 
\ref{pro:prp-fcs:msRN_1opt}, 
so the $\ms{H}^\sMa_\sg$ (\ref{eqn:mat-opt:msHM11g}) are all optimal by construction. 
By the $1$-optimality of $\sHR_N$ (see Proposition \ref{pro:prp-fcs:msRN_1opt})
we have $\ms{H}^\sMa_\sg(\tau,z)=-2+O(q)$ for all $\sg\in\sMa$. 
In light of this the next result amounts to the statement that $\mc{W}_{2,1}^\opt(\sG)_2$ is not empty 
when $\sG=\sMa$.
\begin{pro}\label{pro:mat-opt:WG2nonempty}
There exists a virtual graded $\sMa$-module $W=\bigoplus_D W_D$ such that $\phi^W_\sg=\ms{H}^\sMa_\sg$ for $\sg\in \sMa$.
\end{pro}
\begin{proof}
Recall from \S~\ref{sec:int-mth} that we say that a subgroup $\sMa$ in $2.\sMb$ is intransitive if its image under the composition (\ref{eqn:int-mth:M11toS12}) fixes a point. 
We verify the existence of the $\sMa$-module of the statement of the proposition by restricting the $2.\sMb$-module structure of Proposition \ref{pro:mat-um3:WGalpha2nonempty} to an intransitive copy of $\sMa$ in $2.\sMb$. 
We require to determine how the assignment $\ua:[\sg]\to\rho_\sg$ of characters to $2.\sMb$ restricts to such a subgroup.
These characters are described explicitly in \S~4.8 of \cite{MUM}. From this description we check that $\rho_\sg$ is trivial for every $\sg$ in an intransitive copy of $\sMa$. 
So $W$ as in Proposition \ref{pro:mat-um3:WGalpha2nonempty} restricts along an intransitive copy of $\sMa$ to an $\sMa$-module with the same characters as in Proposition \ref{pro:mat-opt:WG2nonempty}. This completes the proof.
\end{proof}

\begin{rmk}
Since the coefficients of the $\ms{H}^\sMa_\sg$ are rational numbers by construction, they must in fact be rational integers for all $\sg\in \sMa$, and in particular for $o(\sg)=11$, in order for Proposition \ref{pro:mat-opt:WG2nonempty} to be true. This particular consequence of Proposition \ref{pro:mat-opt:WG2nonempty} may be compared to the $N=11$ case of Theorem 4.1.1 in \cite{cncsga}. (See also Corollary 4.2.2 in op.\ cit.)
\end{rmk}
The character table of $\sMa$ is given in Table \ref{tab:chars:m11}. For $\chi$ an irreducible character of $\sMa$ define the multiplicity generating function
\begin{gather}\label{eqn:mat-opt:msHM11chi}
	\ms{H}^\sMa_\chi(\tau):=\sum_{n,s\in\ZZ}m_\chi(W_{s^2-4n})q^{n}y^s
\end{gather} 
where $W_D$ is as in Proposition \ref{pro:mat-opt:WG2nonempty}, and $m_\chi(W_D)$ is the multiplicity of $\chi$ in the $\sMa$-module $W_D$. The first few Fourier coefficients of each of the resulting holomorphic mock Jacobi forms $\ms{H}^\sMa_\chi$ are given in Table \ref{tab:mults:m11}.

\begin{proof}[Proof of Theorem \ref{thm:mat-opt:tor}]
We have $\opc^\opt_{2,1}(\sG)\leq 2$ according to Proposition \ref{pro:mat-opt:WG2nonempty}. 
To see that $\opc^\opt_{2,1}(\sG)\geq 2$ we note, for example, that the space of optimal holomorphic mock Jacobi forms of weight $2$ and index $1$ for $\GammaOJ(5)$ is one-dimensional, spanned by $\sHR_5(\tau,z)=-1+O(q)$ (cf.\ (\ref{eqn:prp-fcs:msRN})), and $\opc=2$ is the smallest positive integer such that $\opc\sHR_5$ has integer coefficients. So $\opc^\opt_{2,1}(\sG)=2$.

It remains to compute $\mc{L}^\opt_{2,1}(\sG)$. For this we note that if $W\in \mc{W}^\opt_{2,1}(\sG)_0$ then, being optimal and having vanishing constant term, $\phi^W_\sg$ must belong to $S_{2,1}(\GammaOJ(o(\sg)))$ for every $\sg\in \sG$. We have $S_{2,1}(\GammaOJ(o(\sg)))=\{0\}$ for $\sg\in \sMa$ except when $o(\sg)=11$, and $S_{2,1}(\GammaOJ(11))_\ZZ$ is generated by $\varphi_{11}$ (as in (\ref{eqn:mat-opt:varphi11})), 
so $\mc{W}^\opt_{2,1}(\sG)_0$ 
is identified with the sublattice 
\begin{gather}\label{eqn:mat-opt:ZZvarphi11otimesRsG11}
\ZZ\varphi_{11}\otimes_\ZZ R(\sG)_{(11)}
\end{gather} 
of $L_{2,1}(\sG)$, where $R(\sG)_{(11)}$ denotes the $\ZZ$-module of virtual characters of $\sMa$ that are supported on elements of order $11$. 

To compute $R(\sG)_{(11)}$ 
let $\delta_{11}$ be the class function on $\sG$ such that $\delta_{11}(\sg)=1$ when $o(\sg)=11$, and $\delta_{11}(\sg)=0$ when $o(\sg)\neq 11$.
Then by applying Thompson's reformulation 
of Brauer's theorem on virtual characters 
(see e.g.\ \cite{MR822245}) to the character table of $\sG$ (see Table \ref{tab:chars:m11})
we may deduce that if 
$x$ is a class function in $R(\sG)_{(11)}$ with rational values then $x$ is an integer multiple of $11\delta_{11}$. 
To understand the rest of $R(\sG)_{(11)}$ we observe that $\chi_6$ and $\chi_7$ (as in Table \ref{tab:chars:m11}) are the only characters of $\sG$ that have non-rational entries on elements of order $11$.
Also, $\chi_6-\chi_7$ vanishes away from elements of order $11$ so is contained in $R(\sG)_{(11)}$. Thus $R(\sG)_{(11)}$ is composed of the (possibly rational) combinations of 
$11\delta_{11}$ 
and 
$\chi_6-\chi_7$ 
that are integer combinations of the irreducible characters $\chi_j$. 
A computation with Table \ref{tab:chars:m11} reveals that 
\begin{gather}\label{eqn:mat-opt:11d}
	11\delta_{11} = 2\chi_1-2\chi_2-2\chi_3-2\chi_4-\chi_6-\chi_7+2\chi_9,
\end{gather}
so $R(\sG)_{(11)}$ is generated, as a $\ZZ$-module, by 
\begin{gather}
\begin{split}\label{eqn:mat-opt:5bx11A5bx11B}
\frac{11}{2}\delta_{11}-\frac12\chi_6+\frac12\chi_7 &= \chi_1-\chi_2-\chi_3-\chi_4-\chi_6+\chi_9,
	\\
\frac{11}{2}\delta_{11}+\frac12\chi_6-\frac12\chi_7 &= \chi_1-\chi_2-\chi_3-\chi_4-\chi_7+\chi_9.
\end{split}	
\end{gather}
Thus we have confirmed 
the description 
(\ref{eqn:mat-opt:tor})
of $\mc{L}^\opt_{2,1}(\sG)$ that appears in the statement of the theorem. 
This completes the proof.
\end{proof}

Our next objective is a classification of optimal $\sMa$-modules for a particular non-trivial assignment of characters. 
For this define $\ua:[\sg]\mapsto (\Gamma_\sg,\rho_\sg)$ by setting $\Gamma_\sg=\GammaOJ(o(\sg))$, and 
\begin{gather}\label{eqn:mat-opt:alpha}
	\rho_\sg\left(
	\gamma,v,\varsigma
	\right)
	:=
	\begin{cases}
		1&\text{if $o(\sg)\not \in \{4,8\}$,}\\
		\ex\left(-\frac{2cd}{o(\sg)^2}\right)&\text{if $o(\sg)\in \{4,8\}$,}
	\end{cases}
\end{gather}
for $(\gamma,(\lambda,\mu),\varsigma) \in \GammaOJ(o(\sg))$ (cf.\ (\ref{eqn:prp-jac:mltGJ})), when 
$\gamma=\left(\begin{smallmatrix}a&b\\c&d\end{smallmatrix}\right)$. 
Also let $\varphi_{8|4}$ denote the unique element of $S_{2,1}(\GammaOJ(32))$ such that 
\begin{gather}\label{eqn:mat-opt:varphi84}
	\varphi_{8|4}(\tau,z)=(y^{-1}+y)q+O(q^2)
\end{gather}
as $\Im(\tau)\to \infty$.
Then actually $\varphi_{8|4}$ belongs to $S_{2,1}(\GammaOJ(8),\rho_{\sg})$ for $\sg\in \sMa$ an element of order $8$ (cf. (\ref{eqn:mat-opt:alpha})), and the statement of the classification of optimal modules for $(\sG,\ua)$ is as follows.
\begin{thm}\label{thm:mat-opt:tra}
For $\sG=\sMa$ and $\ua$ as in (\ref{eqn:mat-opt:alpha}) we have 
$\opc^\opt_{2,1}(\sG,\ua)=2$, 
and $\mc{L}^\opt_{2,1}(\sG,\ua)$ is the sublattice of $L_{2,1}(\sG)$ generated by 
\begin{gather}
	\begin{split}\label{eqn:mat-opt:tra}
	&\varphi_{8|4}\otimes(\chi_3-\chi_4),\\
	&\varphi_{8|4}\otimes(\chi_1-\chi_5-\chi_9+\chi_{10}),\\
	&\varphi_{11}\otimes(\chi_1-\chi_2-\chi_3-\chi_4-\chi_6+\chi_9),\\
	&\varphi_{11}\otimes(\chi_1-\chi_2-\chi_3-\chi_4-\chi_7+\chi_9).
	\end{split}
\end{gather}
\end{thm}
\begin{rmk}\label{rmk:mat-opt:mcL21sGua}
Let $N_{8|4}$ denote the norm of $\varphi_{8|4}$ with respect to the Petersson inner product, and let $N_{11}$ be as in Remark \ref{rmk:mat-opt:mcL21sG}. 
Then from Theorem \ref{thm:mat-opt:tra} we see that $\mc{L}^\opt_{2,1}(\sG,\ua)$ is 
a copy of the lattice structure on $\ZZ^4$ for which 
the 
associated quadratic form is $Q(x_1,x_2,x_3,x_4)=N_{8|4}^2(x_1^2+2x_2^2)+N_{11}^2(3x_3^2+5x_3x_4+3x_4^2)$.
\end{rmk}

To compute 
$\opc^\opt_{2,1}(\sG,\ua)$ 
we first seek a suitable counterpart to Proposition \ref{pro:mat-opt:WG2nonempty}. For this we define 
\begin{gather}\label{eqn:mat-opt:msHM11alphag}
\ms{H}^{(\sMa,\ua)}_\sg:= 
\begin{cases}
	2\sHR_{o(\sg)}&\text{for $o(\sg)\not\in\{4,8,11\}$,}\\
	-2\th_{1,0}^0\left(\th_{4,0}^0-\th_{4,4}^0\right)^2\th_{1,0}&\text{for $o(\sg)=4$,}\\
	-2\th_{1,0}^0\left(\th_{16,0}^0-\th_{16,16}^0\right)^2\th_{1,0}-4\varphi_{8|4}&\text{for $o(\sg)=8$,}\\
	2\sHR_{11}-\frac{11}{5}\varphi_{11}&\text{for $o(\sg)=11$,}
\end{cases}
\end{gather} 
where $\sHR_N$ is as in (\ref{eqn:prp-fcs:msRN}), 
and $\theta_{m,r}^0(\tau):=\theta_{m,r}(\tau,0)$ for $\theta_{m,r}$ as in (\ref{eqn:prp-jac:thetamr}),
and $\varphi_{8|4}$
and $\varphi_{11}$ are as in (\ref{eqn:mat-opt:varphi84}) and (\ref{eqn:mat-opt:varphi11}), respectively. 
The first few Fourier coefficients of each of the $\ms{H}^{(\sMa,\ua)}_\sg$ are given in Table \ref{tab:coeffs:m11t}.

We may check directly that the $\ms{H}^{(\sMa,\ua)}_\sg$ of (\ref{eqn:mat-opt:msHM11alphag}) are all optimal. So the next result states that $\mc{W}_{2,1}^\opt(\sG,\ua)_2$ is not empty, for $\sG=\sMa$ and $\ua$ as in (\ref{eqn:mat-opt:alpha}).
\begin{pro}\label{pro:mat-opt:WGalpha2nonempty}
There exists a virtual graded $\sMa$-module $W=\bigoplus_D W_D$ such that $\phi^W_\sg=\ms{H}^{(\sMa,\ua)}_\sg$ for each $\sg\in\sMa$.
\end{pro}
\begin{proof}
As in the proof of Proposition \ref{pro:mat-opt:WG2nonempty} we verify the existence of the $\sMa$-module in question by restricting the $2.\sMb$-module structure of Proposition \ref{pro:mat-um3:WGalpha2nonempty} to a suitable subgroup $\sMa<2.\sMb$, but in this case we choose a copy of $\sMa$ that maps to a transitive subgroup of ${\sf S}_{12}$ 
under the composition (\ref{eqn:int-mth:M11toS12}). (Such a copy of $\sMa$ is called transitive in \S~\ref{sec:int-mth}.)
Again using the explicit description of the characters $\rho_\sg$ for $\sg\in 2.\sMb$ that appears in \S~4.8 of \cite{MUM} we find that $\rho_\sg$ is trivial for every $\sg$ in a transitive copy of $\sMa$ except for when $o(\sg)\in \{4,8\}$. Moreover, the $\rho_\sg$ for $\sg$ in a transitive copy 
agree with the specification (\ref{eqn:mat-opt:alpha}). So $W$ as in Proposition \ref{pro:mat-um3:WGalpha2nonempty} restricts along a transitive copy of $\sMa$ to an $\sMa$-module with the same characters as in Proposition \ref{pro:mat-opt:WGalpha2nonempty}. 
Inspecting further we find that the trace functions obtained by restriction agree with those in (\ref{eqn:mat-opt:msHM11alphag}) except when $o(\sg)=8$, in which case the difference is $4\varphi_{8|4}$ (cf.\ (\ref{eqn:mat-opt:varphi84})). We finish the proof of Proposition \ref{pro:mat-opt:WGalpha2nonempty} by noting that the coefficients of $\varphi_{8|4}$ are rational integers, and observing that 
\begin{gather}\label{eqn:mat-opt:8d_pre}
	\chi_1 - \chi_5 -\chi_9+\chi_{10}
\end{gather}
is a virtual character that takes the value $4$ on elements of $\sMa$ of order $8$, and vanishes on all other elements. (Cf.\ Table \ref{tab:chars:m11}.)
\end{proof}

For $\chi$ an irreducible character of $\sMa$ define the multiplicity generating function $\ms{H}^{(\sMa,\ua)}_\chi$ by the right-hand side of (\ref{eqn:mat-opt:msHM11chi}) but with $W_D$ as in Proposition \ref{pro:mat-opt:WGalpha2nonempty}. The first few Fourier coefficients of each of the resulting holomorphic mock Jacobi forms $\ms{H}^{(\sMa,\ua)}_\chi$ are given in Table \ref{tab:mults:m11t}.

\begin{proof}[Proof of Theorem \ref{thm:mat-opt:tra}.]
Comparing (\ref{eqn:mat-opt:msHM11alphag}) with (\ref{eqn:mat-opt:msHM11g}) we see that $\ms{H}^\sMa_\sg=\ms{H}^{(\sMa,\ua)}_\sg$ for $o(\sg)\not \in \{4,8\}$. 
We have already seen that the $\ms{H}^\sMa_\sg$ are $2$-optimal, so the $\ms{H}^{(\sMa,\ua)}_\sg$ are $2$-optimal for $o(\sg)\not\in\{4,8\}$. The $2$-optimality of $\ms{H}^{(\sMa,\ua)}_\sg$ in the remaining cases follows from the definition of the theta series $\theta_{m,r}$ (cf.\ (\ref{eqn:prp-jac:thetamr})).
 We also check directly that $\ms{H}^{(\sMa,\ua)}_\sg$ transforms according to the characters specified in (\ref{eqn:mat-opt:alpha}) for $o(\sg)\in \{4,8\}$.
So we have 
$\opc^\opt_{2,1}(\sG,\ua)\leq 2$
according to Proposition \ref{pro:mat-opt:WGalpha2nonempty}. The inequality 
$\opc^\opt_{2,1}(\sG,\ua)\geq 2$
holds for the same reason we gave that 
$\opc^\opt_{2,1}(\sG)\geq 2$ 
in the proof of Theorem \ref{thm:mat-opt:tor}, 
so 
$\opc^\opt_{2,1}(\sG,\ua)=2$.

The computation of $\mc{L}^\opt_{2,1}(\sG,\ua)$ is similar to the computation of $\mc{L}^\opt_{2,1}(\sG)$ that appears in the proof of Theorem \ref{thm:mat-opt:tor}, with the main difference being that $S_{2,1}(\Gamma_\sg,\rho_\sg)_\ZZ$ is now non-zero also when $o(\sg)=8$. Specifically, we have
\begin{gather}
	S_{2,1}(\Gamma_\sg,\rho_\sg)_\ZZ =
	\begin{cases}
		\ZZ\varphi_{8|4}&\text{ for $o(\sg)=8$,}\\
		\ZZ\varphi_{11}&\text{ for $o(\sg)=11$,}\\
		0&\text{ else,}
	\end{cases}
\end{gather}
where $\varphi_{8|4}$ and $\varphi_{11}$ are as in (\ref{eqn:mat-opt:varphi84}) and (\ref{eqn:mat-opt:varphi11}), respectively.

To proceed with the determination of $\mc{L}^\opt_{2,1}(\sG,\ua)$
suppose that $x$ is a class function on $\sMa$ that vanishes on $\sg$ unless $o(\sg)\in\{8,11\}$. 
In this case Thompson's result on virtual characters (cf.\ the proof of Theorem \ref{thm:mat-opt:tor}) tells us that $x$ has integer values if and only if $x\in\ZZ4\delta_{8}+\ZZ11\delta_{11}$, where 
$\delta_{11}$ is as in (\ref{eqn:mat-opt:11d}) and $\delta_8$ is defined analogously.
Now the only irreducible characters of $\sG$ that are non-rational on elements of order $8$ or $11$ are $\chi_3$ and $\chi_4$, being non-rational only on elements of order $8$, and $\chi_6$ and $\chi_7$, being non-rational only on elements of order $11$. 
From this much we may conclude that 
$\mc{W}^\opt_{2,1}(\sG,\ua)_0$ is identified with the sublattice 
\begin{gather}\label{eqn:mat-opt:ZZvarphi84otimesRsG8plus}
\ZZ\varphi_{8|4}\otimes R(\sG)_{(8)} + \ZZ\varphi_{11}\otimes R(\sG)_{(11)}
\end{gather}
of $L_{2,1}(\sG)$, where 
$R(\sG)_{(11)}$ is as in (\ref{eqn:mat-opt:ZZvarphi11otimesRsG11}), and
$R(\sG)_{(8)}$ 
is defined analogously.

We have already computed $R(\sG)_{(11)}$ in the proof of Theorem \ref{thm:mat-opt:tor}. To compute $R(\sG)_{(8)}$ we use Table \ref{tab:chars:m11} to compute
\begin{gather}\label{eqn:mat-opt:8d}
	4\delta_{8} = \chi_1 - \chi_5 -\chi_9+\chi_{10}
\end{gather}
(cf.\ (\ref{eqn:mat-opt:8d_pre}))
and to see that $\chi_3-\chi_4$ vanishes away from elements of order $8$. Thus $R(\sG)_{(8)}$ is generated (as a $\ZZ$-module) by $4\delta_8$ and $\chi_3-\chi_4$, and thus we have 
verified 
the description of $\mc{L}^\opt_{2,1}(\sG,\ua)$ that appears in the statement of the theorem (see (\ref{eqn:mat-opt:tra})). Thus the theorem is proved.
\end{proof}

\subsection{Congruent Numbers}\label{sec:mat-art}

Our objective in this section is to establish a result that relates congruent numbers to the $\sMa$-modules of \S~\ref{sec:mat-opt}.
For the statement we note from Table \ref{tab:chars:m11} that $\sMa$ has a unique irreducible representation of dimension $55$.
Also, we take $\sG=\sMa$ in what follows, and we call to mind the optimal $\sMa$-modules 
$\mc{W}^\opt_{2,1}(\sG,\ua)$
that are classified, according to Proposition \ref{pro:prp-opt:mcWropmcLrop}, by Theorem \ref{thm:mat-opt:tra}.

\begin{thm}\label{thm:mat-art:cngnmb}
Let $D<0$ be square-free and satisfy $D\equiv 21\xmod{24}$. 
If there exists an $\sMa$-module $W=\bigoplus_D W_D$ in 
$\mc{W}^\opt_{2,1}(\sG,\ua)$ 
such that $W_D$ contains an irreducible $55$-dimensional submodule with non-zero multiplicity then $|D|$ is not a congruent number.
\end{thm}

\begin{proof}
Recall the mock Jacobi forms $\ms{H}^{(\sMa,\ua)}_\sg$ of (\ref{eqn:mat-opt:msHM11alphag}). Define $C^{(\sMa,\ua)}_\sg(D)$ for $\sg\in \sMa$ by requiring that $\ms{H}^{(\sMa,\ua)}_\sg(\tau,z)=\sum_{n,\ell}C^{(\sMa,\ua)}_\sg(\ell^2-4n)q^ny^\ell$, and to ease notation set $C_N(D):=C^{(\sMa,\ua)}_\sg(D)$ when $o(\sg)=N$.
Also define $C_{8|4}(D)$ so that $\varphi_{8|4}(\tau,z)=\sum_{n,\ell}C_{8|4}(\ell^2-4n)q^ny^\ell$ where $\varphi_{8|4}$ is as in (\ref{eqn:mat-opt:varphi84}).

Let $W\in 
\mc{W}^\opt_{2,1}(\sG,\ua)$ 
and $D<0$, and write $m_{55}(W_D)$ for the multiplicity of the $55$-dimensional irreducible representation in the $\sMa$-module $W_D$. Then
the description (\ref{eqn:mat-opt:tra}) of $\mc{L}^\opt_{2,1}(\sG,\ua)$ together with
a computation with Table \ref{tab:chars:m11} yields that $m_{55}(W_D)$ is given explicitly by
\begin{gather}\label{eqn:mat-art:m55WD-1}
\frac1{144}C_1(D)
-\frac1{48}C_2(D)
+\frac1{18}C_3(D)
-\frac1{8}C_4(D)
-\frac1{6}C_6(D)
+\frac1{4}C_8(D)
+\frac14\Re(\lambda_{8|4})C_{8|4}(D)
\end{gather}
for some 
$\lambda_{8|4}\in 4\ZZ+2a_2\ZZ$ (where $a_2$ is as in Table \ref{tab:chars:m11}).

Now $C_4(D)$ and $C_8(D)$ vanish for odd $D$, and the remaining $C_N(D)$ in (\ref{eqn:mat-art:m55WD-1}) may be expressed explicitly in terms of Hurwitz class numbers. For example, $C_1(D)=24\HGH(D)$ according to (\ref{eqn:mat-opt:msHM11alphag}), because $\sHR_1=12\sHGH$ (cf.\ (\ref{eqn:prp-fcs:msRN})). For $\sg$ of order $N=o(\sg)\in\{2,3\}$ the mock Jacobi form $\ms{H}^\sMa_\sg$ is a linear combination of $\sHGH$ and the Cohen--Eisenstein series $\sHCE_{o(\sg)}$ (cf.\ (\ref{eqn:prp-cln:msHCohN})), because these two forms span the space $\JJ_{2,1}(\Gamma_0^{\rm J}(N))$ in each case. So the formulas 
(\ref{eqn:prp-cln:HHurND-Nprime}) and (\ref{eqn:prp-cln:HCohND-Nprime})
allow us to write $C_N(D)$ in terms of class numbers, for $N\in \{2,3\}$. For $N=6$ we achieve an expression for $C_6(D)$ in terms of class numbers by observing that $\JJ_{2,1}(\Gamma_0^{\rm J}(6))$ is spanned by $\sHGH$, $\sHCE_2$, $\sHCE_3$, and the Jacobi form $\sum_{n,\ell}C_6'(\ell^2-4n)q^ny^\ell$, where $C_6'(D):=\HGH(36D)-\HGH(D)$. Noting that
\begin{gather}
	C_6(D)=
	\left(
	1
	-2\left(\frac{D}{8}\right)
	-3\left(\frac{D}{3}\right)
	+6\left(\frac{D}{24}\right)
	\right)\HGH(D)
\end{gather}
for $D$ fundamental, we ultimately obtain from (\ref{eqn:mat-art:m55WD-1}) that 
\begin{gather}
m_{55}(W_D)=\left(\frac{D}{3}\right)\left(1-\left(\frac{D}{8}\right)\right)\HGH(D)+\frac14\Re(\lambda_{8|4})C_{8|4}(D), 
\end{gather}
which simplifies to $m_{55}(W_D)=\frac14\Re(\lambda_{8|4})C_{8|4}(D)$ because $D\equiv 0\xmod 3$ by hypothesis.
So for the discriminants $D$ in question, the multiplicity $m_{55}(W_D)$ depends only on the contribution from $\varphi_{8|4}$. 

Let $E$ denote the elliptic curve over $\QQ$ defined by $y^2=x^3-x$, and let $E\otimes D$ denote the $D$-th quadratic twist of $E$, defined by $y^2=x^3-D^2x$. Then, as explained in the introduction of \cite{MR700775}, for example, $|D|$ is a congruent number if and only if $E\otimes D$ has positive rank (i.e. infinitely many rational points). Define $a(n)\in\ZZ$ for $n\geq 1$ by requiring that
\begin{gather}
	\frac{\eta(4\tau)^5\eta(16\tau)}{\eta(2\tau)^2\eta(8\tau)}=\sum_n a(n)q^n,
\end{gather}
where $\eta$ is as in (\ref{eqn:prp-jac:eta}).
Then, by Theorem 3 of \cite{MR700775} we have that $L(E\otimes D,1)=Ca(|D|)^2$ for some non-zero constant $C$. Also, $a(|D|)=C_{8|4}(D)$ for $D\equiv 5\xmod 8$. So for the discriminants $D$ under consideration, if $m_{55}(W_D)$ is non-zero then $L(E\otimes D,1)$ is non-zero, and this implies that $E\otimes D$ has finitely many rational points according to \cite{MR954295}. So $|D|$ is not a congruent number in this case. This completes the proof.
\end{proof}
\begin{rmk}
For Theorem \ref{thm:mat-art:cngnmb} it suffices to consider the $\sMa$-module $W$ of Proposition \ref{pro:mat-opt:WGalpha2nonempty}. Indeed, we see from Table \ref{tab:mults:m11t} that $3$ and $59$ are correctly flagged as non-congruent numbers. 
\end{rmk}
\begin{rmk}
If the Birch--Swinnerton-Dyer conjecture were known to be true (cf.\ \cite{MR2238272}) we would be able to prove a converse to Theorem \ref{thm:mat-art:cngnmb}, and use the $\sMa$-module $W$ to identify congruent numbers. For example, $D=-219$ satisfies the conditions of Theorem \ref{thm:mat-art:cngnmb}, and the $55$-dimensional irreducible representation has vanishing multiplicity in $W_D$. This is consistent with the fact that $219$ is congruent. (We have $219=\frac{ab}{2}$ when $a=\frac{264}{13}$ and $b=\frac{949}{44}$, for example.)
\end{rmk}

\appendix

\section{Characters and Coefficients}\label{tab:chars}

Here we give the character table of the sporadic group $\sMa$, and also some coefficients of the McKay--Thompson series and multiplicity generating functions for the graded infinite-dimensional virtual modules for this group that feature in \S~\ref{sec:mat}.

In Table \ref{tab:chars:m11} we use the notation $a_n:=\sqrt{-n}$ and $b_n:=-\frac12+\frac12\sqrt{-n}$.

\begin{table}[h]
\begin{center}
\begin{small}
\caption{Character table of $\sMa$}\label{tab:chars:m11}
\vspace{-6pt}
\begin{tabular}{c@{\;\;}|r@{\;\;}r@{\;\;}r@{\;\;}r@{\;\;}r@{\;\;}r@{\;\;}r@{\;\;}r@{\;\;}r@{\;\;}r}
\toprule
$[\sg]$&   ${\sf 1A}$&  ${\sf 2A}$&   ${\sf 3A}$&  ${\sf 4A}$&  ${\sf 5A}$& ${\sf  6A}$&  ${\sf 8A}$&  ${\sf 8B}$& ${\sf 11A}$&  ${\sf 11B}$\\
\midrule
${\chi}_{1}$&   $1$&   $1$&   $1$&   $1$&   $1$&   $1$&   $1$&   $1$&   $1$&   $1$\\
${\chi}_{2}$&   $10$&   $2$&   $1$&   $2$&   $0$&   $-1$&   $0$&   $0$&   $-1$&   $-1$\\
${\chi}_{3}$&   $10$&   $-2$&   $1$&   $0$&   $0$&   $1$&   $a_2$&   $-a_2$&   $-1$&   $-1$\\
${\chi}_{4}$&   $10$&   $-2$&   $1$&   $0$&   $0$&   $1$&   $-a_2$&   $a_2$&   $-1$&   $-1$\\
${\chi}_{5}$&   $11$&   $3$&   $2$&   $-1$&   $1$&   $0$&   $-1$&   $-1$&   $0$&   $0$\\
${\chi}_{6}$&   $16$&   $0$&   $-2$&   $0$&   $1$&   $0$&   $0$&   $0$&   $b_{11}$&   $\overline{b_{11}}$\\
${\chi}_{7}$&   $16$&   $0$&   $-2$&   $0$&   $1$&   $0$&   $0$&   $0$&   $\overline{b_{11}}$&   $b_{11}$\\
${\chi}_{8}$&   $44$&   $4$&   $-1$&   $0$&   $-1$&   $1$&   $0$&   $0$&   $0$&   $0$\\
${\chi}_{9}$&   $45$&   $-3$&   $0$&   $1$&   $0$&   $0$&   $-1$&   $-1$&   $1$&   $1$\\
${\chi}_{10}$&   $55$&   $-1$&   $1$&   $-1$&   $0$&   $-1$&   $1$&   $1$&   $0$&   $0$\\
\bottomrule
\end{tabular}
\end{small}
\end{center}
\end{table}

\clearpage

\begin{table}[ht]
\vspace{-8pt}
\begin{footnotesize}
\begin{center}
\caption{McKay--Thompson series $\mathscr{H}^{\sMa}_{\sg}$}\label{tab:coeffs:m11}
\vspace{-6pt}
\begin{tabular}{c|rrrrrrrr}\toprule
$|D|$	&{\sf 1A}	&{\sf 2A}	&{\sf 3A}	&{\sf 4A}	&{\sf 5A}	&{\sf 6A}	&{\sf 8AB}	&{\sf 11AB}	\\
\midrule
0	&$-2$	&$-2$	&$-2$	&$-2$	&$-2$	&$-2$	&$-2$	&$-2$	\\
3	&$8$	&$-8$	&$-1$	&$0$	&$-2$	&$1$	&$0$	&$-3$	\\
4	&$12$	&$-4$	&$-6$	&$-4$	&$2$	&$2$	&$0$	&$1$	\\
7	&$24$	&$8$	&$-12$	&$0$	&$-6$	&$-4$	&$0$	&$2$	\\
8	&$24$	&$-8$	&$6$	&$-8$	&$-6$	&$-2$	&$0$	&$2$	\\
11	&$24$	&$-24$	&$6$	&$0$	&$4$	&$-6$	&$0$	&$2$	\\
12	&$32$	&$0$	&$-4$	&$0$	&$-8$	&$0$	&$-8$	&$-1$	\\
15	&$48$	&$16$	&$-6$	&$0$	&$-2$	&$-2$	&$0$	&$-7$	\\
16	&$36$	&$4$	&$-18$	&$4$	&$6$	&$-2$	&$-4$	&$-8$	\\
19	&$24$	&$-24$	&$-12$	&$0$	&$4$	&$12$	&$0$	&$2$	\\
20	&$48$	&$-16$	&$12$	&$-16$	&$-2$	&$-4$	&$0$	&$-7$	\\
23	&$72$	&$24$	&$18$	&$0$	&$-18$	&$6$	&$0$	&$-5$	\\
24	&$48$	&$-16$	&$-6$	&$-16$	&$8$	&$2$	&$0$	&$4$	\\
27	&$32$	&$-32$	&$5$	&$0$	&$-8$	&$-5$	&$0$	&$-1$	\\
28	&$48$	&$16$	&$-24$	&$16$	&$-12$	&$-8$	&$8$	&$4$	\\
31	&$72$	&$24$	&$-36$	&$0$	&$12$	&$-12$	&$0$	&$-5$	\\
32	&$72$	&$8$	&$18$	&$8$	&$-18$	&$2$	&$-8$	&$6$	\\
35	&$48$	&$-48$	&$12$	&$0$	&$-2$	&$-12$	&$0$	&$4$	\\
36	&$60$	&$-20$	&$6$	&$-20$	&$10$	&$-2$	&$0$	&$-6$	\\
39	&$96$	&$32$	&$-12$	&$0$	&$16$	&$-4$	&$0$	&$8$	\\
40	&$48$	&$-16$	&$-24$	&$-16$	&$-2$	&$8$	&$0$	&$4$	\\
43	&$24$	&$-24$	&$-12$	&$0$	&$-6$	&$12$	&$0$	&$2$	\\
44	&$96$	&$0$	&$24$	&$0$	&$16$	&$0$	&$-24$	&$-3$	\\
47	&$120$	&$40$	&$30$	&$0$	&$-30$	&$10$	&$0$	&$-12$	\\
48	&$80$	&$16$	&$-10$	&$16$	&$-20$	&$-2$	&$0$	&$-8$	\\
51	&$48$	&$-48$	&$-6$	&$0$	&$8$	&$6$	&$0$	&$4$	\\
52	&$48$	&$-16$	&$-24$	&$-16$	&$-12$	&$8$	&$0$	&$4$	\\
55	&$96$	&$32$	&$-48$	&$0$	&$-4$	&$-16$	&$0$	&$-3$	\\
56	&$96$	&$-32$	&$24$	&$-32$	&$16$	&$-8$	&$0$	&$-14$	\\
59	&$72$	&$-72$	&$18$	&$0$	&$12$	&$-18$	&$0$	&$-5$	\\
60	&$96$	&$32$	&$-12$	&$32$	&$-4$	&$-4$	&$16$	&$-3$	\\
63	&$120$	&$40$	&$12$	&$0$	&$-30$	&$4$	&$0$	&$10$	\\
64	&$84$	&$20$	&$-42$	&$20$	&$14$	&$-10$	&$4$	&$-4$	\\
67	&$24$	&$-24$	&$-12$	&$0$	&$-6$	&$12$	&$0$	&$-9$	\\
68	&$96$	&$-32$	&$24$	&$-32$	&$-24$	&$-8$	&$0$	&$8$	\\
71	&$168$	&$56$	&$42$	&$0$	&$28$	&$14$	&$0$	&$-19$	\\
72	&$72$	&$-24$	&$18$	&$-24$	&$-18$	&$-6$	&$0$	&$6$	\\
75	&$56$	&$-56$	&$-7$	&$0$	&$6$	&$7$	&$0$	&$-10$	\\
76	&$96$	&$0$	&$-48$	&$0$	&$16$	&$0$	&$-24$	&$8$	\\
79	&$120$	&$40$	&$-60$	&$0$	&$20$	&$-20$	&$0$	&$10$	\\
80	&$144$	&$16$	&$36$	&$16$	&$-6$	&$4$	&$-16$	&$-10$	\\
83	&$72$	&$-72$	&$18$	&$0$	&$-18$	&$-18$	&$0$	&$6$	\\
84	&$96$	&$-32$	&$-12$	&$-32$	&$16$	&$4$	&$0$	&$8$	\\
87	&$144$	&$48$	&$-18$	&$0$	&$-36$	&$-6$	&$0$	&$12$	\\
88	&$48$	&$-16$	&$-24$	&$-16$	&$-12$	&$8$	&$0$	&$4$	\\
91	&$48$	&$-48$	&$-24$	&$0$	&$8$	&$24$	&$0$	&$4$	\\
92	&$144$	&$48$	&$36$	&$48$	&$-36$	&$12$	&$24$	&$-21$	\\
95	&$192$	&$64$	&$48$	&$0$	&$-8$	&$16$	&$0$	&$16$	\\
96	&$144$	&$16$	&$-18$	&$16$	&$24$	&$-2$	&$-16$	&$12$	\\
99	&$72$	&$-72$	&$18$	&$0$	&$12$	&$-18$	&$0$	&$-5$	\\
100	&$60$	&$-20$	&$-30$	&$-20$	&$10$	&$10$	&$0$	&$-6$	\\
103	&$120$	&$40$	&$-60$	&$0$	&$-30$	&$-20$	&$0$	&$-12$	\\
104	&$144$	&$-48$	&$36$	&$-48$	&$24$	&$-12$	&$0$	&$-10$	\\
107	&$72$	&$-72$	&$18$	&$0$	&$-18$	&$-18$	&$0$	&$6$	\\
108	&$128$	&$0$	&$20$	&$0$	&$-32$	&$0$	&$-32$	&$-15$	\\
 \bottomrule
\end{tabular}
\end{center}
\end{footnotesize}
\end{table}

\begin{table}[ht]
\vspace{-8pt}
\begin{footnotesize}
\begin{center}
\caption{McKay--Thompson series $\ms{H}^{(\sMa,\ua)}_\sg$}\label{tab:coeffs:m11t}
\vspace{-6pt}
\begin{tabular}{c|rrrrrrrr}\toprule
$|D|$	&{\sf 1A}	&{\sf 2A}	&{\sf 3A}	&{\sf 4A}	&{\sf 5A}	&{\sf 6A}	&{\sf 8AB}	&{\sf 11AB}	\\
\midrule
0	&$-2$	&$-2$	&$-2$	&$-2$	&$-2$	&$-2$	&$-2$	&$-2$	\\
3	&$8$	&$-8$	&$-1$	&$0$	&$-2$	&$1$	&$-4$	&$-3$	\\
4	&$12$	&$-4$	&$-6$	&$4$	&$2$	&$2$	&$-4$	&$1$	\\
7	&$24$	&$8$	&$-12$	&$0$	&$-6$	&$-4$	&$0$	&$2$	\\
8	&$24$	&$-8$	&$6$	&$8$	&$-6$	&$-2$	&$0$	&$2$	\\
11	&$24$	&$-24$	&$6$	&$0$	&$4$	&$-6$	&$4$	&$2$	\\
12	&$32$	&$0$	&$-4$	&$-16$	&$-8$	&$0$	&$0$	&$-1$	\\
15	&$48$	&$16$	&$-6$	&$0$	&$-2$	&$-2$	&$0$	&$-7$	\\
16	&$36$	&$4$	&$-18$	&$-12$	&$6$	&$-2$	&$4$	&$-8$	\\
19	&$24$	&$-24$	&$-12$	&$0$	&$4$	&$12$	&$4$	&$2$	\\
20	&$48$	&$-16$	&$12$	&$16$	&$-2$	&$-4$	&$16$	&$-7$	\\
23	&$72$	&$24$	&$18$	&$0$	&$-18$	&$6$	&$0$	&$-5$	\\
24	&$48$	&$-16$	&$-6$	&$16$	&$8$	&$2$	&$0$	&$4$	\\
27	&$32$	&$-32$	&$5$	&$0$	&$-8$	&$-5$	&$0$	&$-1$	\\
28	&$48$	&$16$	&$-24$	&$0$	&$-12$	&$-8$	&$0$	&$4$	\\
31	&$72$	&$24$	&$-36$	&$0$	&$12$	&$-12$	&$0$	&$-5$	\\
32	&$72$	&$8$	&$18$	&$-24$	&$-18$	&$2$	&$8$	&$6$	\\
35	&$48$	&$-48$	&$12$	&$0$	&$-2$	&$-12$	&$8$	&$4$	\\
36	&$60$	&$-20$	&$6$	&$20$	&$10$	&$-2$	&$-20$	&$-6$	\\
39	&$96$	&$32$	&$-12$	&$0$	&$16$	&$-4$	&$0$	&$8$	\\
40	&$48$	&$-16$	&$-24$	&$16$	&$-2$	&$8$	&$0$	&$4$	\\
43	&$24$	&$-24$	&$-12$	&$0$	&$-6$	&$12$	&$-12$	&$2$	\\
44	&$96$	&$0$	&$24$	&$-48$	&$16$	&$0$	&$0$	&$-3$	\\
47	&$120$	&$40$	&$30$	&$0$	&$-30$	&$10$	&$0$	&$-12$	\\
48	&$80$	&$16$	&$-10$	&$-16$	&$-20$	&$-2$	&$-16$	&$-8$	\\
51	&$48$	&$-48$	&$-6$	&$0$	&$8$	&$6$	&$-8$	&$4$	\\
52	&$48$	&$-16$	&$-24$	&$16$	&$-12$	&$8$	&$16$	&$4$	\\
55	&$96$	&$32$	&$-48$	&$0$	&$-4$	&$-16$	&$0$	&$-3$	\\
56	&$96$	&$-32$	&$24$	&$32$	&$16$	&$-8$	&$0$	&$-14$	\\
59	&$72$	&$-72$	&$18$	&$0$	&$12$	&$-18$	&$-4$	&$-5$	\\
60	&$96$	&$32$	&$-12$	&$0$	&$-4$	&$-4$	&$0$	&$-3$	\\
63	&$120$	&$40$	&$12$	&$0$	&$-30$	&$4$	&$0$	&$10$	\\
64	&$84$	&$20$	&$-42$	&$-12$	&$14$	&$-10$	&$-12$	&$-4$	\\
67	&$24$	&$-24$	&$-12$	&$0$	&$-6$	&$12$	&$4$	&$-9$	\\
68	&$96$	&$-32$	&$24$	&$32$	&$-24$	&$-8$	&$-32$	&$8$	\\
71	&$168$	&$56$	&$42$	&$0$	&$28$	&$14$	&$0$	&$-19$	\\
72	&$72$	&$-24$	&$18$	&$24$	&$-18$	&$-6$	&$0$	&$6$	\\
75	&$56$	&$-56$	&$-7$	&$0$	&$6$	&$7$	&$4$	&$-10$	\\
76	&$96$	&$0$	&$-48$	&$-48$	&$16$	&$0$	&$0$	&$8$	\\
79	&$120$	&$40$	&$-60$	&$0$	&$20$	&$-20$	&$0$	&$10$	\\
80	&$144$	&$16$	&$36$	&$-48$	&$-6$	&$4$	&$16$	&$-10$	\\
83	&$72$	&$-72$	&$18$	&$0$	&$-18$	&$-18$	&$-4$	&$6$	\\
84	&$96$	&$-32$	&$-12$	&$32$	&$16$	&$4$	&$32$	&$8$	\\
87	&$144$	&$48$	&$-18$	&$0$	&$-36$	&$-6$	&$0$	&$12$	\\
88	&$48$	&$-16$	&$-24$	&$16$	&$-12$	&$8$	&$0$	&$4$	\\
91	&$48$	&$-48$	&$-24$	&$0$	&$8$	&$24$	&$8$	&$4$	\\
92	&$144$	&$48$	&$36$	&$0$	&$-36$	&$12$	&$0$	&$-21$	\\
95	&$192$	&$64$	&$48$	&$0$	&$-8$	&$16$	&$0$	&$16$	\\
96	&$144$	&$16$	&$-18$	&$-48$	&$24$	&$-2$	&$16$	&$12$	\\
99	&$72$	&$-72$	&$18$	&$0$	&$12$	&$-18$	&$-4$	&$-5$	\\
100	&$60$	&$-20$	&$-30$	&$20$	&$10$	&$10$	&$-20$	&$-6$	\\
103	&$120$	&$40$	&$-60$	&$0$	&$-30$	&$-20$	&$0$	&$-12$	\\
104	&$144$	&$-48$	&$36$	&$48$	&$24$	&$-12$	&$0$	&$-10$	\\
107	&$72$	&$-72$	&$18$	&$0$	&$-18$	&$-18$	&$12$	&$6$	\\
108	&$128$	&$0$	&$20$	&$-64$	&$-32$	&$0$	&$0$	&$-15$	\\
 \bottomrule
\end{tabular}
\end{center}
\end{footnotesize}
\end{table}
\clearpage

\begin{table}[ht]
\vspace{-8pt}
\begin{footnotesize}
\begin{center}
\caption{Multiplicity generating functions $\mathscr{H}^{\sMa}_{\chi}$}\label{tab:mults:m11}
\vspace{-6pt}
\begin{tabular}{c|rrrrrrrrrr}\toprule
$|D|$	&$1$	&$10$	&$10'$	&$\overline{10'}$	&$11$	&$16$	&$\overline{16}$	&$44$	&$45$	&$55$	\\
\midrule
0	&$-2$	&$0$	&$0$	&$0$	&$0$	&$0$	&$0$	&$0$	&$0$	&$0$	\\
3	&$-1$	&$0$	&$1$	&$1$	&$-1$	&$0$	&$0$	&$0$	&$0$	&$0$	\\
4	&$0$	&$-2$	&$0$	&$0$	&$0$	&$1$	&$1$	&$0$	&$0$	&$0$	\\
7	&$-2$	&$0$	&$-2$	&$-2$	&$-2$	&$0$	&$0$	&$2$	&$0$	&$0$	\\
8	&$-2$	&$-2$	&$0$	&$0$	&$0$	&$-2$	&$-2$	&$0$	&$0$	&$2$	\\
11	&$0$	&$0$	&$0$	&$0$	&$0$	&$0$	&$0$	&$-4$	&$2$	&$2$	\\
12	&$-4$	&$0$	&$0$	&$0$	&$0$	&$-1$	&$-1$	&$2$	&$2$	&$-2$	\\
15	&$-2$	&$2$	&$0$	&$0$	&$0$	&$1$	&$1$	&$2$	&$-2$	&$0$	\\
16	&$-2$	&$2$	&$0$	&$0$	&$0$	&$4$	&$4$	&$0$	&$0$	&$-2$	\\
19	&$2$	&$-4$	&$2$	&$2$	&$-2$	&$2$	&$2$	&$0$	&$2$	&$-2$	\\
20	&$-4$	&$-2$	&$2$	&$2$	&$2$	&$-1$	&$-1$	&$-2$	&$-2$	&$4$	\\
23	&$-2$	&$2$	&$2$	&$2$	&$0$	&$-5$	&$-5$	&$6$	&$-2$	&$0$	\\
24	&$0$	&$-6$	&$0$	&$0$	&$2$	&$2$	&$2$	&$-2$	&$0$	&$2$	\\
27	&$-3$	&$0$	&$1$	&$1$	&$-3$	&$-2$	&$-2$	&$-2$	&$2$	&$2$	\\
28	&$0$	&$4$	&$-4$	&$-4$	&$-8$	&$0$	&$0$	&$4$	&$0$	&$0$	\\
31	&$-2$	&$2$	&$-4$	&$-4$	&$0$	&$7$	&$7$	&$0$	&$-2$	&$0$	\\
32	&$-2$	&$2$	&$0$	&$0$	&$0$	&$-6$	&$-6$	&$4$	&$4$	&$-2$	\\
35	&$-2$	&$0$	&$0$	&$0$	&$-2$	&$-2$	&$-2$	&$-6$	&$4$	&$4$	\\
36	&$-2$	&$-4$	&$2$	&$2$	&$4$	&$2$	&$2$	&$-4$	&$-2$	&$4$	\\
39	&$4$	&$0$	&$-4$	&$-4$	&$4$	&$4$	&$4$	&$0$	&$0$	&$0$	\\
40	&$-2$	&$-8$	&$0$	&$0$	&$-2$	&$2$	&$2$	&$2$	&$0$	&$0$	\\
43	&$0$	&$-4$	&$2$	&$2$	&$-4$	&$0$	&$0$	&$2$	&$2$	&$-2$	\\
44	&$-2$	&$2$	&$2$	&$2$	&$12$	&$1$	&$1$	&$-4$	&$6$	&$-4$	\\
47	&$-4$	&$4$	&$4$	&$4$	&$0$	&$-8$	&$-8$	&$10$	&$-4$	&$0$	\\
48	&$-4$	&$6$	&$0$	&$0$	&$-6$	&$-2$	&$-2$	&$6$	&$0$	&$-2$	\\
51	&$2$	&$-4$	&$2$	&$2$	&$-2$	&$2$	&$2$	&$-4$	&$4$	&$0$	\\
52	&$-4$	&$-8$	&$0$	&$0$	&$-4$	&$0$	&$0$	&$4$	&$0$	&$0$	\\
55	&$-6$	&$2$	&$-6$	&$-6$	&$-4$	&$5$	&$5$	&$4$	&$-2$	&$0$	\\
56	&$-4$	&$-4$	&$4$	&$4$	&$8$	&$2$	&$2$	&$-8$	&$-4$	&$8$	\\
59	&$-2$	&$2$	&$2$	&$2$	&$0$	&$1$	&$1$	&$-12$	&$4$	&$6$	\\
60	&$6$	&$10$	&$-2$	&$-2$	&$-8$	&$1$	&$1$	&$4$	&$-2$	&$0$	\\
63	&$-2$	&$0$	&$-2$	&$-2$	&$-2$	&$-8$	&$-8$	&$10$	&$0$	&$0$	\\
64	&$2$	&$6$	&$-4$	&$-4$	&$-4$	&$8$	&$8$	&$0$	&$0$	&$-2$	\\
67	&$-2$	&$-2$	&$4$	&$4$	&$-4$	&$1$	&$1$	&$2$	&$0$	&$-2$	\\
68	&$-8$	&$-8$	&$0$	&$0$	&$0$	&$-8$	&$-8$	&$0$	&$0$	&$8$	\\
71	&$8$	&$6$	&$6$	&$6$	&$14$	&$3$	&$3$	&$0$	&$-6$	&$0$	\\
72	&$-6$	&$-6$	&$0$	&$0$	&$0$	&$-6$	&$-6$	&$0$	&$0$	&$6$	\\
75	&$-1$	&$-2$	&$5$	&$5$	&$-3$	&$3$	&$3$	&$-4$	&$2$	&$0$	\\
76	&$-4$	&$-4$	&$-4$	&$-4$	&$4$	&$8$	&$8$	&$0$	&$8$	&$-8$	\\
79	&$0$	&$0$	&$-10$	&$-10$	&$0$	&$10$	&$10$	&$0$	&$0$	&$0$	\\
80	&$-2$	&$8$	&$4$	&$4$	&$6$	&$-4$	&$-4$	&$2$	&$4$	&$-4$	\\
83	&$-6$	&$0$	&$0$	&$0$	&$-6$	&$-6$	&$-6$	&$-6$	&$6$	&$6$	\\
84	&$0$	&$-12$	&$0$	&$0$	&$4$	&$4$	&$4$	&$-4$	&$0$	&$4$	\\
87	&$-6$	&$0$	&$-6$	&$-6$	&$-6$	&$-6$	&$-6$	&$12$	&$0$	&$0$	\\
88	&$-4$	&$-8$	&$0$	&$0$	&$-4$	&$0$	&$0$	&$4$	&$0$	&$0$	\\
91	&$4$	&$-8$	&$4$	&$4$	&$-4$	&$4$	&$4$	&$0$	&$4$	&$-4$	\\
92	&$6$	&$18$	&$6$	&$6$	&$-12$	&$-9$	&$-9$	&$12$	&$-6$	&$0$	\\
95	&$8$	&$0$	&$0$	&$0$	&$8$	&$-8$	&$-8$	&$8$	&$0$	&$0$	\\
96	&$4$	&$2$	&$-4$	&$-4$	&$6$	&$6$	&$6$	&$-2$	&$8$	&$-6$	\\
99	&$-2$	&$2$	&$2$	&$2$	&$0$	&$1$	&$1$	&$-12$	&$4$	&$6$	\\
100	&$-2$	&$-8$	&$2$	&$2$	&$0$	&$6$	&$6$	&$0$	&$-2$	&$0$	\\
103	&$-14$	&$4$	&$-6$	&$-6$	&$-10$	&$2$	&$2$	&$10$	&$-4$	&$0$	\\
104	&$-4$	&$-8$	&$4$	&$4$	&$12$	&$2$	&$2$	&$-12$	&$-4$	&$12$	\\
107	&$-6$	&$0$	&$0$	&$0$	&$-6$	&$-6$	&$-6$	&$-6$	&$6$	&$6$	\\
108	&$-16$	&$4$	&$4$	&$4$	&$4$	&$-7$	&$-7$	&$6$	&$6$	&$-6$	\\
    \bottomrule
\end{tabular}
\end{center}
\end{footnotesize}
\end{table}

\begin{table}[ht]
\vspace{-8pt}
\begin{footnotesize}
\begin{center}
\caption{Multiplicity generating functions $\mathscr{H}^{(\sMa,\ua)}_{\chi}$}\label{tab:mults:m11t}
\vspace{-6pt}
\begin{tabular}{c|rrrrrrrrrr}\toprule
$|D|$	&$1$	&$10$	&$10'$	&$\overline{10'}$	&$11$	&$16$	&$\overline{16}$	&$44$	&$45$	&$55$	\\
\midrule
0	&$-2$	&$0$	&$0$	&$0$	&$0$	&$0$	&$0$	&$0$	&$0$	&$0$	\\
3	&$-2$	&$0$	&$1$	&$1$	&$0$	&$0$	&$0$	&$0$	&$1$	&$-1$	\\
4	&$0$	&$0$	&$0$	&$0$	&$0$	&$1$	&$1$	&$0$	&$2$	&$-2$	\\
7	&$-2$	&$0$	&$-2$	&$-2$	&$-2$	&$0$	&$0$	&$2$	&$0$	&$0$	\\
8	&$0$	&$2$	&$0$	&$0$	&$-2$	&$-2$	&$-2$	&$0$	&$2$	&$0$	\\
11	&$1$	&$0$	&$0$	&$0$	&$-1$	&$0$	&$0$	&$-4$	&$1$	&$3$	\\
12	&$-4$	&$-4$	&$0$	&$0$	&$0$	&$-1$	&$-1$	&$2$	&$-2$	&$2$	\\
15	&$-2$	&$2$	&$0$	&$0$	&$0$	&$1$	&$1$	&$2$	&$-2$	&$0$	\\
16	&$-2$	&$-2$	&$0$	&$0$	&$0$	&$4$	&$4$	&$0$	&$-4$	&$2$	\\
19	&$3$	&$-4$	&$2$	&$2$	&$-3$	&$2$	&$2$	&$0$	&$1$	&$-1$	\\
20	&$4$	&$6$	&$2$	&$2$	&$-6$	&$-1$	&$-1$	&$-2$	&$-2$	&$4$	\\
23	&$-2$	&$2$	&$2$	&$2$	&$0$	&$-5$	&$-5$	&$6$	&$-2$	&$0$	\\
24	&$4$	&$2$	&$0$	&$0$	&$-2$	&$2$	&$2$	&$-2$	&$4$	&$-2$	\\
27	&$-3$	&$0$	&$1$	&$1$	&$-3$	&$-2$	&$-2$	&$-2$	&$2$	&$2$	\\
28	&$-4$	&$0$	&$-4$	&$-4$	&$-4$	&$0$	&$0$	&$4$	&$0$	&$0$	\\
31	&$-2$	&$2$	&$-4$	&$-4$	&$0$	&$7$	&$7$	&$0$	&$-2$	&$0$	\\
32	&$-2$	&$-6$	&$0$	&$0$	&$0$	&$-6$	&$-6$	&$4$	&$-4$	&$6$	\\
35	&$0$	&$0$	&$0$	&$0$	&$-4$	&$-2$	&$-2$	&$-6$	&$2$	&$6$	\\
36	&$-2$	&$6$	&$2$	&$2$	&$4$	&$2$	&$2$	&$-4$	&$8$	&$-6$	\\
39	&$4$	&$0$	&$-4$	&$-4$	&$4$	&$4$	&$4$	&$0$	&$0$	&$0$	\\
40	&$2$	&$0$	&$0$	&$0$	&$-6$	&$2$	&$2$	&$2$	&$4$	&$-4$	\\
43	&$-3$	&$-4$	&$2$	&$2$	&$-1$	&$0$	&$0$	&$2$	&$5$	&$-5$	\\
44	&$-2$	&$-10$	&$2$	&$2$	&$12$	&$1$	&$1$	&$-4$	&$-6$	&$8$	\\
47	&$-4$	&$4$	&$4$	&$4$	&$0$	&$-8$	&$-8$	&$10$	&$-4$	&$0$	\\
48	&$-12$	&$-2$	&$0$	&$0$	&$2$	&$-2$	&$-2$	&$6$	&$0$	&$-2$	\\
51	&$0$	&$-4$	&$2$	&$2$	&$0$	&$2$	&$2$	&$-4$	&$6$	&$-2$	\\
52	&$4$	&$0$	&$0$	&$0$	&$-12$	&$0$	&$0$	&$4$	&$0$	&$0$	\\
55	&$-6$	&$2$	&$-6$	&$-6$	&$-4$	&$5$	&$5$	&$4$	&$-2$	&$0$	\\
56	&$4$	&$12$	&$4$	&$4$	&$0$	&$2$	&$2$	&$-8$	&$4$	&$0$	\\
59	&$-3$	&$2$	&$2$	&$2$	&$1$	&$1$	&$1$	&$-12$	&$5$	&$5$	\\
60	&$-2$	&$2$	&$-2$	&$-2$	&$0$	&$1$	&$1$	&$4$	&$-2$	&$0$	\\
63	&$-2$	&$0$	&$-2$	&$-2$	&$-2$	&$-8$	&$-8$	&$10$	&$0$	&$0$	\\
64	&$-6$	&$-2$	&$-4$	&$-4$	&$4$	&$8$	&$8$	&$0$	&$0$	&$-2$	\\
67	&$-1$	&$-2$	&$4$	&$4$	&$-5$	&$1$	&$1$	&$2$	&$-1$	&$-1$	\\
68	&$-8$	&$8$	&$0$	&$0$	&$0$	&$-8$	&$-8$	&$0$	&$16$	&$-8$	\\
71	&$8$	&$6$	&$6$	&$6$	&$14$	&$3$	&$3$	&$0$	&$-6$	&$0$	\\
72	&$0$	&$6$	&$0$	&$0$	&$-6$	&$-6$	&$-6$	&$0$	&$6$	&$0$	\\
75	&$0$	&$-2$	&$5$	&$5$	&$-4$	&$3$	&$3$	&$-4$	&$1$	&$1$	\\
76	&$-4$	&$-16$	&$-4$	&$-4$	&$4$	&$8$	&$8$	&$0$	&$-4$	&$4$	\\
79	&$0$	&$0$	&$-10$	&$-10$	&$0$	&$10$	&$10$	&$0$	&$0$	&$0$	\\
80	&$-2$	&$-8$	&$4$	&$4$	&$6$	&$-4$	&$-4$	&$2$	&$-12$	&$12$	\\
83	&$-7$	&$0$	&$0$	&$0$	&$-5$	&$-6$	&$-6$	&$-6$	&$7$	&$5$	\\
84	&$16$	&$4$	&$0$	&$0$	&$-12$	&$4$	&$4$	&$-4$	&$0$	&$4$	\\
87	&$-6$	&$0$	&$-6$	&$-6$	&$-6$	&$-6$	&$-6$	&$12$	&$0$	&$0$	\\
88	&$0$	&$0$	&$0$	&$0$	&$-8$	&$0$	&$0$	&$4$	&$4$	&$-4$	\\
91	&$6$	&$-8$	&$4$	&$4$	&$-6$	&$4$	&$4$	&$0$	&$2$	&$-2$	\\
92	&$-6$	&$6$	&$6$	&$6$	&$0$	&$-9$	&$-9$	&$12$	&$-6$	&$0$	\\
95	&$8$	&$0$	&$0$	&$0$	&$8$	&$-8$	&$-8$	&$8$	&$0$	&$0$	\\
96	&$4$	&$-14$	&$-4$	&$-4$	&$6$	&$6$	&$6$	&$-2$	&$-8$	&$10$	\\
99	&$-3$	&$2$	&$2$	&$2$	&$1$	&$1$	&$1$	&$-12$	&$5$	&$5$	\\
100	&$-2$	&$2$	&$2$	&$2$	&$0$	&$6$	&$6$	&$0$	&$8$	&$-10$	\\
103	&$-14$	&$4$	&$-6$	&$-6$	&$-10$	&$2$	&$2$	&$10$	&$-4$	&$0$	\\
104	&$8$	&$16$	&$4$	&$4$	&$0$	&$2$	&$2$	&$-12$	&$8$	&$0$	\\
107	&$-3$	&$0$	&$0$	&$0$	&$-9$	&$-6$	&$-6$	&$-6$	&$3$	&$9$	\\
108	&$-16$	&$-12$	&$4$	&$4$	&$4$	&$-7$	&$-7$	&$6$	&$-10$	&$10$	\\    
\bottomrule
\end{tabular}
\end{center}
\end{footnotesize}
\end{table}

\clearpage


\setstretch{1.08}
\addcontentsline{toc}{section}{References}
\bibliographystyle{plain}

\end{document}